\documentclass[12pt]{amsart}
\usepackage{amsfonts,latexsym}

\newcommand{\sm}{\setminus}
\newcommand{\J}{{\mathcal J}}
\newcommand{\K}{{\mathcal K}}
\newcommand{\F}{{\mathcal F}}
\newcommand{\CC}{{\mathcal C}}
\newcommand{\R}{{\bf R}}
\newcommand{\Z}{{\bf Z}}
\newcommand{\thm}{\begin{theorem}}
\newcommand{\etheorem}{\end{theorem}}
\newcommand{\prop}{\begin{proposition}}
\newcommand{\eprop}{\end{proposition}}
\newtheorem{theorem}{Theorem}
\newtheorem{proposition}{Proposition}
\newtheorem{lemma}{Lemma}

\theoremstyle{definition}
\newtheorem{definition}{Definition}
\newtheorem{example}{Example}

\theoremstyle{remark}
\newtheorem{remark}{Remark}

\numberwithin{equation}{section}

\begin{document}

\title{On Finite Order Invariants of Triple Points Free
Plane Curves}

\author{V.A.~Vassiliev}
\address{Steklov Mathematical Inst.,
8 Gubkina str., Moscow 117966, GSP-1/Russia and Independent Moscow University}

\email{vva@mi.ras.ru}

\thanks{Supported in part by INTAS (project 96-0713) and
RFBR (project 98-01-00555)}


\date{Revised version published in 1999}

\dedicatory{To D.~B.~Fuchs with love}

\begin{abstract}
We describe some regular techniques of calculating finite-order invariants of
triple points free smooth plane curves $S^1 \to \R^2$. They are a direct analog
of similar techniques for knot invariants and are based on the calculus of {\em
triangular diagrams} and {\em connected hypergraphs} in the same way as the
calculation of knot invariants is based on the study of chord diagrams and
connected graphs.

E.g., the simplest such invariant is of order 4 and corresponds to the
triangular diagram \unitlength=1.00mm \special{em:linewidth 0.4pt}
\linethickness{0.4pt}
\begin{picture}(3.33,4.34)
\put(0.67,1.67){\circle{5.33}} \put(0.67,-1.00){\line(-3,5){2.33}}
\put(-1.66,3.00){\line(1,0){4.33}} \put(2.67,3.00){\line(-1,-2){2.00}}
\put(0.67,4.34){\line(-1,-2){2.00}} \put(-1.33,0.34){\line(1,0){4.00}}
\put(2.67,0.34){\line(-1,2){2.00}}
\end{picture}
in the same way as the simplest knot invariant (of order 2) corresponds to the
2-chord diagram $\bigoplus$. Also, following V.~I.~Arnold and other authors we
consider invariants of {\em immersed} triple points free curves and describe
similar techniques also for this problem, and, more generally, for the
calculation of homology groups of the space of immersed plane curves without
points of multiplicity $\ge k$ for any $k \ge 3.$
\end{abstract}

\maketitle

\section*{Introduction}

The intensive study of invariants of generic immersions $S^1 \to \R^2$ was
started by V.~I.~Arnold in \cite{Arnold-20} and continued in
\cite{Arnold-MIAN}, \cite{Arnold-21}, \cite{Viro}, \cite{Ai}, \cite{Tabach},
\cite{Shum-1}, \cite{P}, \cite{Merx-4}, etc.

The most interesting invariant of such objects, the {\em strangeness}, is in
fact an invariant of triple points free immersions $S^1 \to \R^2$ (with allowed
self-tangencies).

Almost simultaneously, \cite{V-11}, \cite{V-94}, I considered the {\em
ornaments}, i.e. collections of plane curves (maybe with singularities) without
intersections of three different components, and developed regular techniques
for calculating their invariants. The present work is the (promised in
\cite{V-11}) substitution of these methods into the theory of triple points
free plane curves.

Below we describe a natural filtration of invariants by their {\em orders}, and
a regular method of calculating all invariants of all finite orders for triple
points free plane curves. Following the idea of \cite{A2}, we reduce the study
of invariants (and other cohomology classes of the space of generic objects) to
that of the homology groups of the complementary {\em discriminant set} of
objects with forbidden singularities (i.e., in our case, of curves with triple
points). The more technical tools of this method are the simplicial resolutions
of discriminants (see \cite{V-2}) and the (arising from them) calculus of {\em
triangular diagrams} and {\em connected hypergraphs}, which are analogs of
chord diagrams and connected graphs arising in the theory of finite-order knot
invariants.

The simplest such invariants are described in following two theorems.

First, as was proposed in \cite{V-11} (see Problem 2 of \S~9 there), we
consider the space of all plane curves $\phi: S^1 \to \R^2$ having no triple
points and no singularities obtained as degenerations of triple points (i.e.,
either the double points at which one of two local branches has a singular
point with $\phi' = 0$, or the points at which $\phi' = \phi'' = 0.)$ The
problem of classifying such objects (called the {\em doodles}) is in the same
relation with the classification of ornaments, in which the isotopy
classification of links is with the homotopy classification.

In this setting, the curves {\large ``$0$''} and {\large ``$8$''} become
equivalent, and the strangeness fails to be an invariant of such objects; this
is an analog of the fact that the trivial chord diagram $\ominus$ does not
define a knot invariant.

\thm \label{doodinv} There are no invariants of doodles of orders 1, 2 or 3,
and there is exactly one invariant of order 4. \etheorem

\begin{figure}
\begin{center}
\unitlength 1.00mm \linethickness{0.4pt}
\begin{picture}(21.00,21.67)
\multiput(1.00,11.00)(0.08,1.17){2}{\line(0,1){1.17}}
\multiput(1.16,13.34)(0.12,0.51){4}{\line(0,1){0.51}}
\multiput(1.63,15.38)(0.11,0.25){7}{\line(0,1){0.25}}
\multiput(2.41,17.09)(0.11,0.14){10}{\line(0,1){0.14}}
\multiput(3.50,18.50)(0.14,0.11){10}{\line(1,0){0.14}}
\multiput(4.91,19.59)(0.25,0.11){7}{\line(1,0){0.25}}
\multiput(6.63,20.38)(0.51,0.12){4}{\line(1,0){0.51}}
\multiput(8.66,20.84)(1.17,0.08){2}{\line(1,0){1.17}}
\multiput(11.00,21.00)(1.17,-0.08){2}{\line(1,0){1.17}}
\multiput(13.34,20.84)(0.51,-0.12){4}{\line(1,0){0.51}}
\multiput(15.38,20.38)(0.25,-0.11){7}{\line(1,0){0.25}}
\multiput(17.09,19.59)(0.14,-0.11){10}{\line(1,0){0.14}}
\multiput(18.50,18.50)(0.11,-0.14){10}{\line(0,-1){0.14}}
\multiput(19.59,17.09)(0.11,-0.25){7}{\line(0,-1){0.25}}
\multiput(20.38,15.38)(0.12,-0.51){4}{\line(0,-1){0.51}}
\multiput(20.84,13.34)(0.08,-1.17){2}{\line(0,-1){1.17}}
\multiput(21.00,11.00)(-0.08,-1.17){2}{\line(0,-1){1.17}}
\multiput(20.84,8.66)(-0.12,-0.51){4}{\line(0,-1){0.51}}
\multiput(20.38,6.63)(-0.11,-0.25){7}{\line(0,-1){0.25}}
\multiput(19.59,4.91)(-0.11,-0.14){10}{\line(0,-1){0.14}}
\multiput(18.50,3.50)(-0.14,-0.11){10}{\line(-1,0){0.14}}
\multiput(17.09,2.41)(-0.25,-0.11){7}{\line(-1,0){0.25}}
\multiput(15.38,1.63)(-0.51,-0.12){4}{\line(-1,0){0.51}}
\multiput(13.34,1.16)(-1.17,-0.08){2}{\line(-1,0){1.17}}
\multiput(11.00,1.00)(-1.17,0.08){2}{\line(-1,0){1.17}}
\multiput(8.66,1.16)(-0.51,0.12){4}{\line(-1,0){0.51}}
\multiput(6.63,1.63)(-0.25,0.11){7}{\line(-1,0){0.25}}
\multiput(4.91,2.41)(-0.14,0.11){10}{\line(-1,0){0.14}}
\multiput(3.50,3.50)(-0.11,0.14){10}{\line(0,1){0.14}}
\multiput(2.41,4.91)(-0.11,0.25){7}{\line(0,1){0.25}}
\multiput(1.63,6.63)(-0.12,0.51){4}{\line(0,1){0.51}}
\multiput(1.16,8.66)(-0.08,1.17){2}{\line(0,1){1.17}}
\put(2.33,6.00){\line(1,0){4.67}} \put(9.00,6.00){\line(1,0){10.33}}
\multiput(19.33,6.00)(-0.12,0.24){17}{\line(0,1){0.24}}
\multiput(16.00,12.00)(-0.12,0.21){42}{\line(0,1){0.21}}
\multiput(11.00,21.00)(-0.12,-0.22){20}{\line(0,-1){0.22}}
\multiput(7.67,15.33)(-0.12,-0.21){45}{\line(0,-1){0.21}}
\multiput(11.00,1.00)(0.12,0.20){20}{\line(0,1){0.20}}
\multiput(14.33,6.67)(0.12,0.21){45}{\line(0,1){0.21}}
\put(19.67,16.00){\line(-1,0){5.00}} \put(13.00,16.00){\line(-1,0){10.67}}
\multiput(2.33,16.00)(0.12,-0.22){20}{\line(0,-1){0.22}}
\multiput(5.67,10.00)(0.12,-0.20){45}{\line(0,-1){0.20}}
\put(2.00,6.00){\circle*{1.33}} \put(2.00,16.00){\circle*{1.33}}
\put(11.00,21.33){\circle*{1.33}} \put(20.00,16.00){\circle*{1.33}}
\put(20.00,6.00){\circle*{1.33}} \put(11.00,1.00){\circle*{1.33}}
\end{picture}
\end{center}
\caption{Unique invariant of order 4 for doodles} \label{md}
\end{figure}
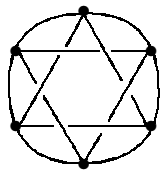

This invariant is depicted by the triangular diagram shown in Fig.~\ref{md}.

(This diagram is an adequate analog of the chord diagram $\bigoplus$ defining
the first nontrivial knot invariant: they both are simplest diagrams of
corresponding kinds, not containing elements with neighboring vertices.)

\begin{figure}
\begin{center}
\unitlength 1.00mm \linethickness{0.4pt}
\begin{picture}(40.00,40.00)
\multiput(20.00,40.00)(1.17,-0.08){2}{\line(1,0){1.17}}
\multiput(22.34,39.84)(0.51,-0.12){4}{\line(1,0){0.51}}
\multiput(24.38,39.38)(0.25,-0.11){7}{\line(1,0){0.25}}
\multiput(26.09,38.59)(0.14,-0.11){10}{\line(1,0){0.14}}
\multiput(27.50,37.50)(0.11,-0.14){10}{\line(0,-1){0.14}}
\multiput(28.59,36.09)(0.11,-0.25){7}{\line(0,-1){0.25}}
\multiput(29.38,34.38)(0.12,-0.51){4}{\line(0,-1){0.51}}
\multiput(29.84,32.34)(0.08,-1.17){2}{\line(0,-1){1.17}}
\multiput(20.00,40.00)(-1.17,-0.08){2}{\line(-1,0){1.17}}
\multiput(17.66,39.84)(-0.51,-0.12){4}{\line(-1,0){0.51}}
\multiput(15.63,39.38)(-0.25,-0.11){7}{\line(-1,0){0.25}}
\multiput(13.91,38.59)(-0.14,-0.11){10}{\line(-1,0){0.14}}
\multiput(12.50,37.50)(-0.11,-0.14){10}{\line(0,-1){0.14}}
\multiput(11.41,36.09)(-0.11,-0.25){7}{\line(0,-1){0.25}}
\multiput(10.63,34.38)(-0.12,-0.51){4}{\line(0,-1){0.51}}
\multiput(10.16,32.34)(-0.08,-1.17){2}{\line(0,-1){1.17}}
\multiput(10.00,30.00)(-1.17,-0.08){2}{\line(-1,0){1.17}}
\multiput(7.66,29.84)(-0.51,-0.12){4}{\line(-1,0){0.51}}
\multiput(5.63,29.38)(-0.25,-0.11){7}{\line(-1,0){0.25}}
\multiput(3.91,28.59)(-0.14,-0.11){10}{\line(-1,0){0.14}}
\multiput(2.50,27.50)(-0.11,-0.14){10}{\line(0,-1){0.14}}
\multiput(1.41,26.09)(-0.11,-0.25){7}{\line(0,-1){0.25}}
\multiput(0.63,24.38)(-0.12,-0.51){4}{\line(0,-1){0.51}}
\multiput(0.16,22.34)(-0.08,-1.17){2}{\line(0,-1){1.17}}
\multiput(0.00,20.00)(0.08,-1.17){2}{\line(0,-1){1.17}}
\multiput(0.16,17.66)(0.12,-0.51){4}{\line(0,-1){0.51}}
\multiput(0.63,15.63)(0.11,-0.25){7}{\line(0,-1){0.25}}
\multiput(1.41,13.91)(0.11,-0.14){10}{\line(0,-1){0.14}}
\multiput(2.50,12.50)(0.14,-0.11){10}{\line(1,0){0.14}}
\multiput(3.91,11.41)(0.25,-0.11){7}{\line(1,0){0.25}}
\multiput(5.63,10.63)(0.51,-0.12){4}{\line(1,0){0.51}}
\multiput(7.66,10.16)(1.17,-0.08){2}{\line(1,0){1.17}}
\multiput(10.00,10.00)(0.08,-1.17){2}{\line(0,-1){1.17}}
\multiput(10.16,7.66)(0.12,-0.51){4}{\line(0,-1){0.51}}
\multiput(10.63,5.63)(0.11,-0.25){7}{\line(0,-1){0.25}}
\multiput(11.41,3.91)(0.11,-0.14){10}{\line(0,-1){0.14}}
\multiput(12.50,2.50)(0.14,-0.11){10}{\line(1,0){0.14}}
\multiput(13.91,1.41)(0.25,-0.11){7}{\line(1,0){0.25}}
\multiput(15.63,0.63)(0.51,-0.12){4}{\line(1,0){0.51}}
\multiput(17.66,0.16)(1.17,-0.08){2}{\line(1,0){1.17}}
\multiput(20.00,0.00)(1.17,0.08){2}{\line(1,0){1.17}}
\multiput(22.34,0.16)(0.51,0.12){4}{\line(1,0){0.51}}
\multiput(24.38,0.63)(0.25,0.11){7}{\line(1,0){0.25}}
\multiput(26.09,1.41)(0.14,0.11){10}{\line(1,0){0.14}}
\multiput(27.50,2.50)(0.11,0.14){10}{\line(0,1){0.14}}
\multiput(28.59,3.91)(0.11,0.25){7}{\line(0,1){0.25}}
\multiput(29.38,5.63)(0.12,0.51){4}{\line(0,1){0.51}}
\multiput(29.84,7.66)(0.08,1.17){2}{\line(0,1){1.17}}
\multiput(30.00,10.00)(1.17,0.08){2}{\line(1,0){1.17}}
\multiput(32.34,10.16)(0.51,0.12){4}{\line(1,0){0.51}}
\multiput(34.38,10.63)(0.25,0.11){7}{\line(1,0){0.25}}
\multiput(36.09,11.41)(0.14,0.11){10}{\line(1,0){0.14}}
\multiput(37.50,12.50)(0.11,0.14){10}{\line(0,1){0.14}}
\multiput(38.59,13.91)(0.11,0.25){7}{\line(0,1){0.25}}
\multiput(39.38,15.63)(0.12,0.51){4}{\line(0,1){0.51}}
\multiput(39.84,17.66)(0.08,1.17){2}{\line(0,1){1.17}}
\multiput(40.00,20.00)(-0.08,1.17){2}{\line(0,1){1.17}}
\multiput(39.84,22.34)(-0.12,0.51){4}{\line(0,1){0.51}}
\multiput(39.38,24.38)(-0.11,0.25){7}{\line(0,1){0.25}}
\multiput(38.59,26.09)(-0.11,0.14){10}{\line(0,1){0.14}}
\multiput(37.50,27.50)(-0.14,0.11){10}{\line(-1,0){0.14}}
\multiput(36.09,28.59)(-0.25,0.11){7}{\line(-1,0){0.25}}
\multiput(34.38,29.38)(-0.51,0.12){4}{\line(-1,0){0.51}}
\multiput(32.34,29.84)(-1.17,0.08){2}{\line(-1,0){1.17}}
\put(10.00,10.00){\line(1,0){2.42}}
\multiput(12.42,10.08)(1.13,0.12){2}{\line(1,0){1.13}}
\multiput(14.69,10.31)(0.53,0.10){4}{\line(1,0){0.53}}
\multiput(16.80,10.70)(0.39,0.11){5}{\line(1,0){0.39}}
\multiput(18.75,11.25)(0.30,0.12){6}{\line(1,0){0.30}}
\multiput(20.55,11.95)(0.21,0.11){8}{\line(1,0){0.21}}
\multiput(22.19,12.81)(0.16,0.11){9}{\line(1,0){0.16}}
\multiput(23.67,13.83)(0.13,0.12){10}{\line(1,0){0.13}}
\multiput(25.00,15.00)(0.12,0.13){10}{\line(0,1){0.13}}
\multiput(26.17,16.33)(0.11,0.16){9}{\line(0,1){0.16}}
\multiput(27.19,17.81)(0.11,0.21){8}{\line(0,1){0.21}}
\multiput(28.05,19.45)(0.12,0.30){6}{\line(0,1){0.30}}
\multiput(28.75,21.25)(0.11,0.39){5}{\line(0,1){0.39}}
\multiput(29.30,23.20)(0.10,0.53){4}{\line(0,1){0.53}}
\multiput(29.69,25.31)(0.12,1.13){2}{\line(0,1){1.13}}
\put(29.92,27.58){\line(0,1){2.42}} \put(30.00,10.00){\line(0,1){2.42}}
\multiput(29.92,12.42)(-0.12,1.13){2}{\line(0,1){1.13}}
\multiput(29.69,14.69)(-0.10,0.53){4}{\line(0,1){0.53}}
\multiput(29.30,16.80)(-0.11,0.39){5}{\line(0,1){0.39}}
\multiput(28.75,18.75)(-0.12,0.30){6}{\line(0,1){0.30}}
\multiput(28.05,20.55)(-0.11,0.21){8}{\line(0,1){0.21}}
\multiput(27.19,22.19)(-0.11,0.16){9}{\line(0,1){0.16}}
\multiput(26.17,23.67)(-0.12,0.13){10}{\line(0,1){0.13}}
\multiput(25.00,25.00)(-0.13,0.12){10}{\line(-1,0){0.13}}
\multiput(23.67,26.17)(-0.16,0.11){9}{\line(-1,0){0.16}}
\multiput(22.19,27.19)(-0.21,0.11){8}{\line(-1,0){0.21}}
\multiput(20.55,28.05)(-0.30,0.12){6}{\line(-1,0){0.30}}
\multiput(18.75,28.75)(-0.39,0.11){5}{\line(-1,0){0.39}}
\multiput(16.80,29.30)(-0.53,0.10){4}{\line(-1,0){0.53}}
\multiput(14.69,29.69)(-1.13,0.12){2}{\line(-1,0){1.13}}
\put(12.42,29.92){\line(-1,0){2.42}} \put(30.00,30.00){\line(-1,0){2.42}}
\multiput(27.58,29.92)(-1.13,-0.12){2}{\line(-1,0){1.13}}
\multiput(25.31,29.69)(-0.53,-0.10){4}{\line(-1,0){0.53}}
\multiput(23.20,29.30)(-0.39,-0.11){5}{\line(-1,0){0.39}}
\multiput(21.25,28.75)(-0.30,-0.12){6}{\line(-1,0){0.30}}
\multiput(19.45,28.05)(-0.21,-0.11){8}{\line(-1,0){0.21}}
\multiput(17.81,27.19)(-0.16,-0.11){9}{\line(-1,0){0.16}}
\multiput(16.33,26.17)(-0.13,-0.12){10}{\line(-1,0){0.13}}
\multiput(15.00,25.00)(-0.12,-0.13){10}{\line(0,-1){0.13}}
\multiput(13.83,23.67)(-0.11,-0.16){9}{\line(0,-1){0.16}}
\multiput(12.81,22.19)(-0.11,-0.21){8}{\line(0,-1){0.21}}
\multiput(11.95,20.55)(-0.12,-0.30){6}{\line(0,-1){0.30}}
\multiput(11.25,18.75)(-0.11,-0.39){5}{\line(0,-1){0.39}}
\multiput(10.70,16.80)(-0.10,-0.53){4}{\line(0,-1){0.53}}
\multiput(10.31,14.69)(-0.12,-1.13){2}{\line(0,-1){1.13}}
\put(10.08,12.42){\line(0,-1){2.42}} \put(10.00,30.00){\line(0,-1){2.42}}
\multiput(10.08,27.58)(0.12,-1.13){2}{\line(0,-1){1.13}}
\multiput(10.31,25.31)(0.10,-0.53){4}{\line(0,-1){0.53}}
\multiput(10.70,23.20)(0.11,-0.39){5}{\line(0,-1){0.39}}
\multiput(11.25,21.25)(0.12,-0.30){6}{\line(0,-1){0.30}}
\multiput(11.95,19.45)(0.11,-0.21){8}{\line(0,-1){0.21}}
\multiput(12.81,17.81)(0.11,-0.16){9}{\line(0,-1){0.16}}
\multiput(13.83,16.33)(0.12,-0.13){10}{\line(0,-1){0.13}}
\multiput(15.00,15.00)(0.13,-0.12){10}{\line(1,0){0.13}}
\multiput(16.33,13.83)(0.16,-0.11){9}{\line(1,0){0.16}}
\multiput(17.81,12.81)(0.21,-0.11){8}{\line(1,0){0.21}}
\multiput(19.45,11.95)(0.30,-0.12){6}{\line(1,0){0.30}}
\multiput(21.25,11.25)(0.39,-0.11){5}{\line(1,0){0.39}}
\multiput(23.20,10.70)(0.53,-0.10){4}{\line(1,0){0.53}}
\multiput(25.31,10.31)(1.13,-0.12){2}{\line(1,0){1.13}}
\put(27.58,10.08){\line(1,0){2.42}}
\end{picture}
\end{center}
\caption{A plane curve not equivalent to the circle in the space of singular
triple points free curves (after A.~B.~Merkov)} \label{merx}
\end{figure}
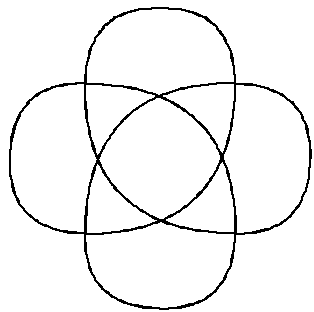

This invariant proves, in particular, that the curve of Fig.~\ref{merx}
(discovered previously by A.~B.~Merkov) is not equivalent to a circle.
\medskip

Further, following \cite{Arnold-20}, let us consider the immersed curves in
$\R^2.$

\thm \label{idoodinv} There are only the following invariants of orders $\le 4$
of triple points free plane {\bf immersed} curves $S^1 \to \R^2$:

1) no invariants of order 1;

2) one invariant of order 2 (the Arnold's {\em strangeness}; by some reasons we
denote this invariant by the simplest "triangular diagram" of Fig.~\ref{one}a;

3) one more invariant of order 3 (its natural notation see in Fig.~\ref{one}b;

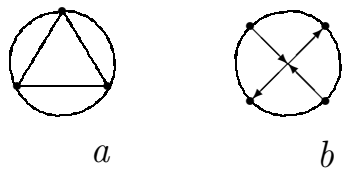
\begin{figure}
\begin{center}
\special{em:linewidth 0.4pt} \unitlength 1.00mm \linethickness{0.4pt}
\begin{picture}(46.00,21.67)
\multiput(10.00,21.00)(0.79,-0.09){2}{\line(1,0){0.79}}
\multiput(11.59,20.82)(0.30,-0.11){5}{\line(1,0){0.30}}
\multiput(13.09,20.28)(0.17,-0.11){8}{\line(1,0){0.17}}
\multiput(14.43,19.42)(0.11,-0.11){10}{\line(0,-1){0.11}}
\multiput(15.54,18.27)(0.12,-0.20){7}{\line(0,-1){0.20}}
\multiput(16.37,16.91)(0.10,-0.30){5}{\line(0,-1){0.30}}
\multiput(16.86,15.39)(0.07,-0.80){2}{\line(0,-1){0.80}}
\multiput(17.00,13.80)(-0.11,-0.79){2}{\line(0,-1){0.79}}
\multiput(16.77,12.22)(-0.12,-0.30){5}{\line(0,-1){0.30}}
\multiput(16.19,10.73)(-0.11,-0.16){8}{\line(0,-1){0.16}}
\multiput(15.29,9.42)(-0.13,-0.12){9}{\line(-1,0){0.13}}
\multiput(14.11,8.34)(-0.20,-0.11){7}{\line(-1,0){0.20}}
\multiput(12.72,7.55)(-0.38,-0.11){4}{\line(-1,0){0.38}}
\put(11.19,7.10){\line(-1,0){1.59}}
\multiput(9.60,7.01)(-0.52,0.09){3}{\line(-1,0){0.52}}
\multiput(8.03,7.28)(-0.25,0.10){6}{\line(-1,0){0.25}}
\multiput(6.56,7.90)(-0.16,0.12){8}{\line(-1,0){0.16}}
\multiput(5.27,8.84)(-0.12,0.13){9}{\line(0,1){0.13}}
\multiput(4.22,10.05)(-0.11,0.20){7}{\line(0,1){0.20}}
\multiput(3.48,11.46)(-0.10,0.39){4}{\line(0,1){0.39}}
\put(3.07,13.00){\line(0,1){1.60}}
\multiput(3.03,14.60)(0.11,0.52){3}{\line(0,1){0.52}}
\multiput(3.34,16.16)(0.11,0.24){6}{\line(0,1){0.24}}
\multiput(4.01,17.61)(0.11,0.14){9}{\line(0,1){0.14}}
\multiput(4.98,18.88)(0.14,0.11){9}{\line(1,0){0.14}}
\multiput(6.22,19.89)(0.24,0.12){6}{\line(1,0){0.24}}
\multiput(7.65,20.59)(0.59,0.10){4}{\line(1,0){0.59}}
\put(4.00,11.00){\line(1,0){12.00}}
\multiput(16.00,11.00)(-0.12,0.20){50}{\line(0,1){0.20}}
\multiput(10.00,21.00)(-0.12,-0.20){50}{\line(0,-1){0.20}}
\multiput(40.00,21.00)(0.79,-0.09){2}{\line(1,0){0.79}}
\multiput(41.59,20.82)(0.30,-0.11){5}{\line(1,0){0.30}}
\multiput(43.09,20.28)(0.17,-0.11){8}{\line(1,0){0.17}}
\multiput(44.43,19.42)(0.11,-0.11){10}{\line(0,-1){0.11}}
\multiput(45.54,18.27)(0.12,-0.20){7}{\line(0,-1){0.20}}
\multiput(46.37,16.91)(0.10,-0.30){5}{\line(0,-1){0.30}}
\multiput(46.86,15.39)(0.07,-0.80){2}{\line(0,-1){0.80}}
\multiput(47.00,13.80)(-0.11,-0.79){2}{\line(0,-1){0.79}}
\multiput(46.77,12.22)(-0.12,-0.30){5}{\line(0,-1){0.30}}
\multiput(46.19,10.73)(-0.11,-0.16){8}{\line(0,-1){0.16}}
\multiput(45.29,9.42)(-0.13,-0.12){9}{\line(-1,0){0.13}}
\multiput(44.11,8.34)(-0.20,-0.11){7}{\line(-1,0){0.20}}
\multiput(42.72,7.55)(-0.38,-0.11){4}{\line(-1,0){0.38}}
\put(41.19,7.10){\line(-1,0){1.59}}
\multiput(39.60,7.01)(-0.52,0.09){3}{\line(-1,0){0.52}}
\multiput(38.03,7.28)(-0.25,0.10){6}{\line(-1,0){0.25}}
\multiput(36.56,7.90)(-0.16,0.12){8}{\line(-1,0){0.16}}
\multiput(35.27,8.84)(-0.12,0.13){9}{\line(0,1){0.13}}
\multiput(34.22,10.05)(-0.11,0.20){7}{\line(0,1){0.20}}
\multiput(33.48,11.46)(-0.10,0.39){4}{\line(0,1){0.39}}
\put(33.07,13.00){\line(0,1){1.60}}
\multiput(33.03,14.60)(0.11,0.52){3}{\line(0,1){0.52}}
\multiput(33.34,16.16)(0.11,0.24){6}{\line(0,1){0.24}}
\multiput(34.01,17.61)(0.11,0.14){9}{\line(0,1){0.14}}
\multiput(34.98,18.88)(0.14,0.11){9}{\line(1,0){0.14}}
\multiput(36.22,19.89)(0.24,0.12){6}{\line(1,0){0.24}}
\multiput(37.65,20.59)(0.59,0.10){4}{\line(1,0){0.59}}
\put(45.00,2.00){\makebox(0,0)[cc]{{\large b}}}
\put(15.00,2.00){\makebox(0,0)[cc]{{\large a}}}
\put(40.00,14.00){\vector(1,1){4.67}} \put(44.67,9.33){\vector(-1,1){4.67}}
\put(40.00,14.00){\vector(-1,-1){4.67}} \put(35.33,18.67){\vector(1,-1){4.67}}
\put(4.00,11.00){\circle*{1.00}} \put(16.00,11.00){\circle*{1.00}}
\put(10.00,21.00){\circle*{1.00}} \put(35.00,19.00){\circle*{1.00}}
\put(35.00,9.00){\circle*{1.00}} \put(45.00,9.00){\circle*{1.00}}
\put(45.00,19.00){\circle*{1.00}}
\end{picture}
\end{center}
\caption{Notation for invariants of order 2 and 3 of immersed curves}
\label{one}
\end{figure}

4) five more invariants of order 4 (they are described in Fig.~\ref{two}).
\etheorem

\begin{figure}
\begin{center}
\unitlength 1.00mm \linethickness{0.4pt}
\begin{picture}(125.00,25.67)
%\circle(8.00,15.00){14.00}
\multiput(8.00,22.00)(0.79,-0.09){2}{\line(1,0){0.79}}
\multiput(9.59,21.82)(0.30,-0.11){5}{\line(1,0){0.30}}
\multiput(11.09,21.28)(0.17,-0.11){8}{\line(1,0){0.17}}
\multiput(12.43,20.42)(0.11,-0.11){10}{\line(0,-1){0.11}}
\multiput(13.54,19.27)(0.12,-0.20){7}{\line(0,-1){0.20}}
\multiput(14.37,17.91)(0.10,-0.30){5}{\line(0,-1){0.30}}
\multiput(14.86,16.39)(0.07,-0.80){2}{\line(0,-1){0.80}}
\multiput(15.00,14.80)(-0.11,-0.79){2}{\line(0,-1){0.79}}
\multiput(14.77,13.22)(-0.12,-0.30){5}{\line(0,-1){0.30}}
\multiput(14.19,11.73)(-0.11,-0.16){8}{\line(0,-1){0.16}}
\multiput(13.29,10.42)(-0.13,-0.12){9}{\line(-1,0){0.13}}
\multiput(12.11,9.34)(-0.20,-0.11){7}{\line(-1,0){0.20}}
\multiput(10.72,8.55)(-0.38,-0.11){4}{\line(-1,0){0.38}}
\put(9.19,8.10){\line(-1,0){1.59}}
\multiput(7.60,8.01)(-0.52,0.09){3}{\line(-1,0){0.52}}
\multiput(6.03,8.28)(-0.25,0.10){6}{\line(-1,0){0.25}}
\multiput(4.56,8.90)(-0.16,0.12){8}{\line(-1,0){0.16}}
\multiput(3.27,9.84)(-0.12,0.13){9}{\line(0,1){0.13}}
\multiput(2.22,11.05)(-0.11,0.20){7}{\line(0,1){0.20}}
\multiput(1.48,12.46)(-0.10,0.39){4}{\line(0,1){0.39}}
\put(1.07,14.00){\line(0,1){1.60}}
\multiput(1.03,15.60)(0.11,0.52){3}{\line(0,1){0.52}}
\multiput(1.34,17.16)(0.11,0.24){6}{\line(0,1){0.24}}
\multiput(2.01,18.61)(0.11,0.14){9}{\line(0,1){0.14}}
\multiput(2.98,19.88)(0.14,0.11){9}{\line(1,0){0.14}}
\multiput(4.22,20.89)(0.24,0.12){6}{\line(1,0){0.24}}
\multiput(5.65,21.59)(0.59,0.10){4}{\line(1,0){0.59}}
%\end
%\emline(2.00,18.00)(14.00,18.00)
\put(2.00,18.00){\line(1,0){12.00}}
%\end
%\emline(14.00,18.00)(8.00,22.00)
\multiput(14.00,18.00)(-0.18,0.12){34}{\line(-1,0){0.18}}
%\end
%\emline(8.00,22.00)(2.00,18.00)
\multiput(8.00,22.00)(-0.18,-0.12){34}{\line(-1,0){0.18}}
%\end
%\emline(2.00,12.00)(14.00,12.00)
\put(2.00,12.00){\line(1,0){12.00}}
%\end
%\emline(14.00,12.00)(8.00,8.00)
\multiput(14.00,12.00)(-0.18,-0.12){34}{\line(-1,0){0.18}}
%\end
%\emline(8.00,8.00)(2.00,12.00)
\multiput(8.00,8.00)(-0.18,0.12){34}{\line(-1,0){0.18}}
%\end
%\circle(28.00,15.00){14.00}
\multiput(28.00,22.00)(0.79,-0.09){2}{\line(1,0){0.79}}
\multiput(29.59,21.82)(0.30,-0.11){5}{\line(1,0){0.30}}
\multiput(31.09,21.28)(0.17,-0.11){8}{\line(1,0){0.17}}
\multiput(32.43,20.42)(0.11,-0.11){10}{\line(0,-1){0.11}}
\multiput(33.54,19.27)(0.12,-0.20){7}{\line(0,-1){0.20}}
\multiput(34.37,17.91)(0.10,-0.30){5}{\line(0,-1){0.30}}
\multiput(34.86,16.39)(0.07,-0.80){2}{\line(0,-1){0.80}}
\multiput(35.00,14.80)(-0.11,-0.79){2}{\line(0,-1){0.79}}
\multiput(34.77,13.22)(-0.12,-0.30){5}{\line(0,-1){0.30}}
\multiput(34.19,11.73)(-0.11,-0.16){8}{\line(0,-1){0.16}}
\multiput(33.29,10.42)(-0.13,-0.12){9}{\line(-1,0){0.13}}
\multiput(32.11,9.34)(-0.20,-0.11){7}{\line(-1,0){0.20}}
\multiput(30.72,8.55)(-0.38,-0.11){4}{\line(-1,0){0.38}}
\put(29.19,8.10){\line(-1,0){1.59}}
\multiput(27.60,8.01)(-0.52,0.09){3}{\line(-1,0){0.52}}
\multiput(26.03,8.28)(-0.25,0.10){6}{\line(-1,0){0.25}}
\multiput(24.56,8.90)(-0.16,0.12){8}{\line(-1,0){0.16}}
\multiput(23.27,9.84)(-0.12,0.13){9}{\line(0,1){0.13}}
\multiput(22.22,11.05)(-0.11,0.20){7}{\line(0,1){0.20}}
\multiput(21.48,12.46)(-0.10,0.39){4}{\line(0,1){0.39}}
\put(21.07,14.00){\line(0,1){1.60}}
\multiput(21.03,15.60)(0.11,0.52){3}{\line(0,1){0.52}}
\multiput(21.34,17.16)(0.11,0.24){6}{\line(0,1){0.24}}
\multiput(22.01,18.61)(0.11,0.14){9}{\line(0,1){0.14}}
\multiput(22.98,19.88)(0.14,0.11){9}{\line(1,0){0.14}}
\multiput(24.22,20.89)(0.24,0.12){6}{\line(1,0){0.24}}
\multiput(25.65,21.59)(0.59,0.10){4}{\line(1,0){0.59}}
%\end
%\emline(28.00,22.00)(22.00,18.00)
\multiput(28.00,22.00)(-0.18,-0.12){34}{\line(-1,0){0.18}}
%\end
%\emline(22.00,18.00)(34.00,12.00)
\multiput(22.00,18.00)(0.24,-0.12){50}{\line(1,0){0.24}}
%\end
%\emline(34.00,12.00)(28.00,22.00)
\multiput(34.00,12.00)(-0.12,0.20){50}{\line(0,1){0.20}}
%\end
%\emline(34.00,18.00)(28.00,8.00)
\multiput(34.00,18.00)(-0.12,-0.20){50}{\line(0,-1){0.20}}
%\end
%\emline(28.00,8.00)(22.00,12.00)
\multiput(28.00,8.00)(-0.18,0.12){34}{\line(-1,0){0.18}}
%\end
%\emline(22.00,12.00)(34.00,18.00)
\multiput(22.00,12.00)(0.24,0.12){50}{\line(1,0){0.24}}
%\end
%\circle(48.00,15.00){14.00}
\multiput(48.00,22.00)(0.79,-0.09){2}{\line(1,0){0.79}}
\multiput(49.59,21.82)(0.30,-0.11){5}{\line(1,0){0.30}}
\multiput(51.09,21.28)(0.17,-0.11){8}{\line(1,0){0.17}}
\multiput(52.43,20.42)(0.11,-0.11){10}{\line(0,-1){0.11}}
\multiput(53.54,19.27)(0.12,-0.20){7}{\line(0,-1){0.20}}
\multiput(54.37,17.91)(0.10,-0.30){5}{\line(0,-1){0.30}}
\multiput(54.86,16.39)(0.07,-0.80){2}{\line(0,-1){0.80}}
\multiput(55.00,14.80)(-0.11,-0.79){2}{\line(0,-1){0.79}}
\multiput(54.77,13.22)(-0.12,-0.30){5}{\line(0,-1){0.30}}
\multiput(54.19,11.73)(-0.11,-0.16){8}{\line(0,-1){0.16}}
\multiput(53.29,10.42)(-0.13,-0.12){9}{\line(-1,0){0.13}}
\multiput(52.11,9.34)(-0.20,-0.11){7}{\line(-1,0){0.20}}
\multiput(50.72,8.55)(-0.38,-0.11){4}{\line(-1,0){0.38}}
\put(49.19,8.10){\line(-1,0){1.59}}
\multiput(47.60,8.01)(-0.52,0.09){3}{\line(-1,0){0.52}}
\multiput(46.03,8.28)(-0.25,0.10){6}{\line(-1,0){0.25}}
\multiput(44.56,8.90)(-0.16,0.12){8}{\line(-1,0){0.16}}
\multiput(43.27,9.84)(-0.12,0.13){9}{\line(0,1){0.13}}
\multiput(42.22,11.05)(-0.11,0.20){7}{\line(0,1){0.20}}
\multiput(41.48,12.46)(-0.10,0.39){4}{\line(0,1){0.39}}
\put(41.07,14.00){\line(0,1){1.60}}
\multiput(41.03,15.60)(0.11,0.52){3}{\line(0,1){0.52}}
\multiput(41.34,17.16)(0.11,0.24){6}{\line(0,1){0.24}}
\multiput(42.01,18.61)(0.11,0.14){9}{\line(0,1){0.14}}
\multiput(42.98,19.88)(0.14,0.11){9}{\line(1,0){0.14}}
\multiput(44.22,20.89)(0.24,0.12){6}{\line(1,0){0.24}}
\multiput(45.65,21.59)(0.59,0.10){4}{\line(1,0){0.59}}
%\end
%\emline(48.00,22.00)(42.00,12.00)
\multiput(48.00,22.00)(-0.12,-0.20){50}{\line(0,-1){0.20}}
%\end
%\emline(42.00,12.00)(54.00,12.00)
\put(42.00,12.00){\line(1,0){12.00}}
%\end
%\emline(54.00,12.00)(48.00,22.00)
\multiput(54.00,12.00)(-0.12,0.20){50}{\line(0,1){0.20}}
%\end
%\emline(42.00,18.00)(54.00,18.00)
\put(42.00,18.00){\line(1,0){12.00}}
%\end
%\emline(54.00,18.00)(48.00,8.00)
\multiput(54.00,18.00)(-0.12,-0.20){50}{\line(0,-1){0.20}}
%\end
%\emline(48.00,8.00)(42.00,18.00)
\multiput(48.00,8.00)(-0.12,0.20){50}{\line(0,1){0.20}}
%\end
\put(28.00,2.00){\makebox(0,0)[cc]{{\large a}}}
%\bezier{80}(77.00,15.00)(77.00,25.00)(87.00,25.00)
\multiput(77.00,15.00)(0.08,1.17){2}{\line(0,1){1.17}}
\multiput(77.16,17.34)(0.12,0.51){4}{\line(0,1){0.51}}
\multiput(77.63,19.38)(0.11,0.25){7}{\line(0,1){0.25}}
\multiput(78.41,21.09)(0.11,0.14){10}{\line(0,1){0.14}}
\multiput(79.50,22.50)(0.14,0.11){10}{\line(1,0){0.14}}
\multiput(80.91,23.59)(0.25,0.11){7}{\line(1,0){0.25}}
\multiput(82.63,24.38)(0.51,0.12){4}{\line(1,0){0.51}}
\multiput(84.66,24.84)(1.17,0.08){2}{\line(1,0){1.17}}
%\end
%\bezier{80}(87.00,25.00)(97.00,25.00)(97.00,15.00)
\multiput(87.00,25.00)(1.17,-0.08){2}{\line(1,0){1.17}}
\multiput(89.34,24.84)(0.51,-0.12){4}{\line(1,0){0.51}}
\multiput(91.38,24.38)(0.25,-0.11){7}{\line(1,0){0.25}}
\multiput(93.09,23.59)(0.14,-0.11){10}{\line(1,0){0.14}}
\multiput(94.50,22.50)(0.11,-0.14){10}{\line(0,-1){0.14}}
\multiput(95.59,21.09)(0.11,-0.25){7}{\line(0,-1){0.25}}
\multiput(96.38,19.38)(0.12,-0.51){4}{\line(0,-1){0.51}}
\multiput(96.84,17.34)(0.08,-1.17){2}{\line(0,-1){1.17}}
%\end
%\bezier{80}(97.00,15.00)(97.00,5.00)(87.00,5.00)
\multiput(97.00,15.00)(-0.08,-1.17){2}{\line(0,-1){1.17}}
\multiput(96.84,12.66)(-0.12,-0.51){4}{\line(0,-1){0.51}}
\multiput(96.38,10.63)(-0.11,-0.25){7}{\line(0,-1){0.25}}
\multiput(95.59,8.91)(-0.11,-0.14){10}{\line(0,-1){0.14}}
\multiput(94.50,7.50)(-0.14,-0.11){10}{\line(-1,0){0.14}}
\multiput(93.09,6.41)(-0.25,-0.11){7}{\line(-1,0){0.25}}
\multiput(91.38,5.63)(-0.51,-0.12){4}{\line(-1,0){0.51}}
\multiput(89.34,5.16)(-1.17,-0.08){2}{\line(-1,0){1.17}}
%\end
%\bezier{80}(87.00,5.00)(77.00,5.00)(77.00,15.00)
\multiput(87.00,5.00)(-1.17,0.08){2}{\line(-1,0){1.17}}
\multiput(84.66,5.16)(-0.51,0.12){4}{\line(-1,0){0.51}}
\multiput(82.63,5.63)(-0.25,0.11){7}{\line(-1,0){0.25}}
\multiput(80.91,6.41)(-0.14,0.11){10}{\line(-1,0){0.14}}
\multiput(79.50,7.50)(-0.11,0.14){10}{\line(0,1){0.14}}
\multiput(78.41,8.91)(-0.11,0.25){7}{\line(0,1){0.25}}
\multiput(77.63,10.63)(-0.12,0.51){4}{\line(0,1){0.51}}
\multiput(77.16,12.66)(-0.08,1.17){2}{\line(0,1){1.17}}
%\end
\put(77.67,18.35){\circle*{1.33}} \put(96.33,18.35){\circle*{1.33}}
\put(87.00,25.00){\circle*{1.33}} \put(80.80,7.05){\circle*{1.33}}
\put(93.20,7.05){\circle*{1.33}}
%\bezier{80}(105.00,15.00)(105.00,25.00)(115.00,25.00)
\multiput(105.00,15.00)(0.08,1.17){2}{\line(0,1){1.17}}
\multiput(105.16,17.34)(0.12,0.51){4}{\line(0,1){0.51}}
\multiput(105.63,19.38)(0.11,0.25){7}{\line(0,1){0.25}}
\multiput(106.41,21.09)(0.11,0.14){10}{\line(0,1){0.14}}
\multiput(107.50,22.50)(0.14,0.11){10}{\line(1,0){0.14}}
\multiput(108.91,23.59)(0.25,0.11){7}{\line(1,0){0.25}}
\multiput(110.63,24.38)(0.51,0.12){4}{\line(1,0){0.51}}
\multiput(112.66,24.84)(1.17,0.08){2}{\line(1,0){1.17}}
%\end
%\bezier{80}(115.00,25.00)(125.00,25.00)(125.00,15.00)
\multiput(115.00,25.00)(1.17,-0.08){2}{\line(1,0){1.17}}
\multiput(117.34,24.84)(0.51,-0.12){4}{\line(1,0){0.51}}
\multiput(119.38,24.38)(0.25,-0.11){7}{\line(1,0){0.25}}
\multiput(121.09,23.59)(0.14,-0.11){10}{\line(1,0){0.14}}
\multiput(122.50,22.50)(0.11,-0.14){10}{\line(0,-1){0.14}}
\multiput(123.59,21.09)(0.11,-0.25){7}{\line(0,-1){0.25}}
\multiput(124.38,19.38)(0.12,-0.51){4}{\line(0,-1){0.51}}
\multiput(124.84,17.34)(0.08,-1.17){2}{\line(0,-1){1.17}}
%\end
%\bezier{80}(125.00,15.00)(125.00,5.00)(115.00,5.00)
\multiput(125.00,15.00)(-0.08,-1.17){2}{\line(0,-1){1.17}}
\multiput(124.84,12.66)(-0.12,-0.51){4}{\line(0,-1){0.51}}
\multiput(124.38,10.63)(-0.11,-0.25){7}{\line(0,-1){0.25}}
\multiput(123.59,8.91)(-0.11,-0.14){10}{\line(0,-1){0.14}}
\multiput(122.50,7.50)(-0.14,-0.11){10}{\line(-1,0){0.14}}
\multiput(121.09,6.41)(-0.25,-0.11){7}{\line(-1,0){0.25}}
\multiput(119.38,5.63)(-0.51,-0.12){4}{\line(-1,0){0.51}}
\multiput(117.34,5.16)(-1.17,-0.08){2}{\line(-1,0){1.17}}
%\end
%\bezier{80}(115.00,5.00)(105.00,5.00)(105.00,15.00)
\multiput(115.00,5.00)(-1.17,0.08){2}{\line(-1,0){1.17}}
\multiput(112.66,5.16)(-0.51,0.12){4}{\line(-1,0){0.51}}
\multiput(110.63,5.63)(-0.25,0.11){7}{\line(-1,0){0.25}}
\multiput(108.91,6.41)(-0.14,0.11){10}{\line(-1,0){0.14}}
\multiput(107.50,7.50)(-0.11,0.14){10}{\line(0,1){0.14}}
\multiput(106.41,8.91)(-0.11,0.25){7}{\line(0,1){0.25}}
\multiput(105.63,10.63)(-0.12,0.51){4}{\line(0,1){0.51}}
\multiput(105.16,12.66)(-0.08,1.17){2}{\line(0,1){1.17}}
%\end
\put(106.00,18.67){\circle*{1.33}} \put(124.00,18.67){\circle*{1.33}}
\put(115.00,25.00){\circle*{1.33}} \put(108.33,7.33){\circle*{1.33}}
\put(122.00,7.33){\circle*{1.33}} \put(87.00,25.00){\vector(1,-3){6.00}}
\put(93.00,6.67){\vector(-4,3){15.33}} \put(77.67,18.33){\vector(1,0){18.33}}
\put(96.00,18.33){\vector(-4,-3){15.33}} \put(80.67,6.67){\vector(1,3){6.00}}
\put(115.00,25.00){\vector(3,-2){9.33}}
\put(124.33,18.67){\vector(-1,-4){3.00}}
\put(121.33,7.33){\vector(-1,0){13.00}} \put(108.33,7.33){\vector(-1,4){3.00}}
\put(105.33,18.67){\vector(3,2){9.67}} \put(87.00,25.00){\vector(1,-3){5.00}}
\put(93.00,6.67){\vector(-4,3){12.33}} \put(77.67,18.33){\vector(1,0){15.33}}
\put(96.00,18.33){\vector(-4,-3){12.33}} \put(80.67,6.67){\vector(1,3){5.00}}
\put(115.00,25.00){\vector(3,-2){7.00}}
\put(124.33,18.67){\vector(-1,-4){2.33}}
\put(121.33,7.33){\vector(-1,0){11.00}} \put(108.33,7.33){\vector(-1,4){2.33}}
\put(105.33,18.67){\vector(3,2){9.67}}
\put(101.00,2.00){\makebox(0,0)[cc]{{\large b}}}
\put(8.00,22.00){\circle*{1.00}} \put(14.00,18.00){\circle*{1.00}}
\put(2.00,18.00){\circle*{1.00}} \put(2.00,12.00){\circle*{1.00}}
\put(14.00,12.00){\circle*{1.00}} \put(22.00,12.00){\circle*{1.00}}
\put(22.00,18.00){\circle*{1.00}} \put(28.00,22.00){\circle*{1.00}}
\put(28.00,8.00){\circle*{1.00}} \put(34.00,12.00){\circle*{1.00}}
\put(34.00,18.00){\circle*{1.00}} \put(42.00,18.00){\circle*{1.00}}
\put(42.00,12.00){\circle*{1.00}} \put(54.00,12.00){\circle*{1.00}}
\put(54.00,18.00){\circle*{1.00}} \put(48.00,22.00){\circle*{1.00}}
\put(48.00,8.00){\circle*{1.00}} \put(8.00,8.00){\circle*{1.00}}
\put(106.00,19.00){\vector(3,2){6.00}}
\end{picture}
\end{center}
\caption{Notation for invariants of order 4} \label{two}
\end{figure}
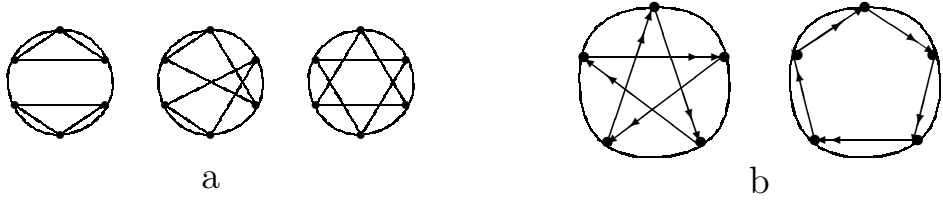

Our methods allow us to calculate also some higher-dimensional cohomology
classes of spaces of $k$-points free plane curves (both immersed or just
$C^\infty$-pa\-ra\-met\-ri\-zed) for any $k \ge 3.$

E.g., let $k=4.$ The set $\Sigma 4$ of all curves with 4-fold selfintersections
has codimension 2 in the space of all plane curves, thus the first interesting
problem is the calculation of the 1-dimensional cohomology group of the
complementary space of immersed plane curves without such points\footnote{The
problem of calculating such homology groups, posed by V.~I.~Arnold (see
\cite{Arnold-probl}, problem 1996-2) forced me to write this paper}.

This cohomology group $H^1(Imm(S^1, \R^2) \sm \Sigma 4)$ also has a natural
filtration, so that the {\em orders} of (some) its elements are well-defined.

\thm \label{4ptfree} For any connected component of the space of immersions
$S^1 \to \R^2$, the first few groups ${\mathcal F}_d$ of order $d$
1-dimensional integer cohomology classes of the space of four-points free
immersions lying in this component are as follows: $\F_1=\F_2 = \F_3 =\F_4=0,$
$\F_5 \simeq \Z^2.$

However, if we calculate the $\Z_2$-cohomology, then we have $\F_3 \simeq \Z_2,
$ $\F_4/\F_3 =0, $ $\F_5/\F_4 \simeq \Z^2_2,$ and if we calculate the
$\Z_5$-cohomology, then $\F_3=0, \F_4 \simeq \Z_5, \F_5/\F_4 \simeq \Z_5^2.$
\etheorem

Two generators of the group $\F_5/\F_4$ (with integer coefficients) are
naturally depicted by two chains shown in Fig.~\ref{hexagon}.

\begin{figure}
\begin{center}
\unitlength 1.00mm \linethickness{0.4pt}
\begin{picture}(61.00,22.67)
%\bezier{80}(1.00,12.00)(1.00,22.00)(11.00,22.00)
\multiput(1.00,12.00)(0.08,1.17){2}{\line(0,1){1.17}}
\multiput(1.16,14.34)(0.12,0.51){4}{\line(0,1){0.51}}
\multiput(1.63,16.38)(0.11,0.25){7}{\line(0,1){0.25}}
\multiput(2.41,18.09)(0.11,0.14){10}{\line(0,1){0.14}}
\multiput(3.50,19.50)(0.14,0.11){10}{\line(1,0){0.14}}
\multiput(4.91,20.59)(0.25,0.11){7}{\line(1,0){0.25}}
\multiput(6.63,21.38)(0.51,0.12){4}{\line(1,0){0.51}}
\multiput(8.66,21.84)(1.17,0.08){2}{\line(1,0){1.17}}
%\end
%\bezier{80}(11.00,22.00)(21.00,22.00)(21.00,12.00)
\multiput(11.00,22.00)(1.17,-0.08){2}{\line(1,0){1.17}}
\multiput(13.34,21.84)(0.51,-0.12){4}{\line(1,0){0.51}}
\multiput(15.38,21.38)(0.25,-0.11){7}{\line(1,0){0.25}}
\multiput(17.09,20.59)(0.14,-0.11){10}{\line(1,0){0.14}}
\multiput(18.50,19.50)(0.11,-0.14){10}{\line(0,-1){0.14}}
\multiput(19.59,18.09)(0.11,-0.25){7}{\line(0,-1){0.25}}
\multiput(20.38,16.38)(0.12,-0.51){4}{\line(0,-1){0.51}}
\multiput(20.84,14.34)(0.08,-1.17){2}{\line(0,-1){1.17}}
%\end
%\bezier{80}(21.00,12.00)(21.00,2.00)(11.00,2.00)
\multiput(21.00,12.00)(-0.08,-1.17){2}{\line(0,-1){1.17}}
\multiput(20.84,9.66)(-0.12,-0.51){4}{\line(0,-1){0.51}}
\multiput(20.38,7.63)(-0.11,-0.25){7}{\line(0,-1){0.25}}
\multiput(19.59,5.91)(-0.11,-0.14){10}{\line(0,-1){0.14}}
\multiput(18.50,4.50)(-0.14,-0.11){10}{\line(-1,0){0.14}}
\multiput(17.09,3.41)(-0.25,-0.11){7}{\line(-1,0){0.25}}
\multiput(15.38,2.63)(-0.51,-0.12){4}{\line(-1,0){0.51}}
\multiput(13.34,2.16)(-1.17,-0.08){2}{\line(-1,0){1.17}}
%\end
%\bezier{80}(11.00,2.00)(1.00,2.00)(1.00,12.00)
\multiput(11.00,2.00)(-1.17,0.08){2}{\line(-1,0){1.17}}
\multiput(8.66,2.16)(-0.51,0.12){4}{\line(-1,0){0.51}}
\multiput(6.63,2.63)(-0.25,0.11){7}{\line(-1,0){0.25}}
\multiput(4.91,3.41)(-0.14,0.11){10}{\line(-1,0){0.14}}
\multiput(3.50,4.50)(-0.11,0.14){10}{\line(0,1){0.14}}
\multiput(2.41,5.91)(-0.11,0.25){7}{\line(0,1){0.25}}
\multiput(1.63,7.63)(-0.12,0.51){4}{\line(0,1){0.51}}
\multiput(1.16,9.66)(-0.08,1.17){2}{\line(0,1){1.17}}
%\end
\put(2.33,7.00){\vector(1,0){17.00}} \put(19.33,7.00){\vector(-1,2){7.67}}
\put(11.67,22.00){\vector(-2,-3){10.00}} \put(2.33,17.00){\vector(1,0){17.33}}
\put(19.67,17.00){\vector(-1,-2){7.67}} \put(12.00,2.00){\vector(-2,3){10.00}}
\put(2.33,7.00){\vector(1,0){15.00}} \put(19.33,7.00){\vector(-1,2){6.67}}
\put(11.67,22.00){\vector(-2,-3){8.67}} \put(2.33,17.00){\vector(1,0){15.33}}
\put(19.67,17.00){\vector(-1,-2){6.67}} \put(12.00,2.00){\vector(-2,3){8.67}}
%\bezier{80}(41.00,12.00)(41.00,22.00)(51.00,22.00)
\multiput(41.00,12.00)(0.08,1.17){2}{\line(0,1){1.17}}
\multiput(41.16,14.34)(0.12,0.51){4}{\line(0,1){0.51}}
\multiput(41.63,16.38)(0.11,0.25){7}{\line(0,1){0.25}}
\multiput(42.41,18.09)(0.11,0.14){10}{\line(0,1){0.14}}
\multiput(43.50,19.50)(0.14,0.11){10}{\line(1,0){0.14}}
\multiput(44.91,20.59)(0.25,0.11){7}{\line(1,0){0.25}}
\multiput(46.63,21.38)(0.51,0.12){4}{\line(1,0){0.51}}
\multiput(48.66,21.84)(1.17,0.08){2}{\line(1,0){1.17}}
%\end
%\bezier{80}(51.00,22.00)(61.00,22.00)(61.00,12.00)
\multiput(51.00,22.00)(1.17,-0.08){2}{\line(1,0){1.17}}
\multiput(53.34,21.84)(0.51,-0.12){4}{\line(1,0){0.51}}
\multiput(55.38,21.38)(0.25,-0.11){7}{\line(1,0){0.25}}
\multiput(57.09,20.59)(0.14,-0.11){10}{\line(1,0){0.14}}
\multiput(58.50,19.50)(0.11,-0.14){10}{\line(0,-1){0.14}}
\multiput(59.59,18.09)(0.11,-0.25){7}{\line(0,-1){0.25}}
\multiput(60.38,16.38)(0.12,-0.51){4}{\line(0,-1){0.51}}
\multiput(60.84,14.34)(0.08,-1.17){2}{\line(0,-1){1.17}}
%\end
%\bezier{80}(61.00,12.00)(61.00,2.00)(51.00,2.00)
\multiput(61.00,12.00)(-0.08,-1.17){2}{\line(0,-1){1.17}}
\multiput(60.84,9.66)(-0.12,-0.51){4}{\line(0,-1){0.51}}
\multiput(60.38,7.63)(-0.11,-0.25){7}{\line(0,-1){0.25}}
\multiput(59.59,5.91)(-0.11,-0.14){10}{\line(0,-1){0.14}}
\multiput(58.50,4.50)(-0.14,-0.11){10}{\line(-1,0){0.14}}
\multiput(57.09,3.41)(-0.25,-0.11){7}{\line(-1,0){0.25}}
\multiput(55.38,2.63)(-0.51,-0.12){4}{\line(-1,0){0.51}}
\multiput(53.34,2.16)(-1.17,-0.08){2}{\line(-1,0){1.17}}
%\end
%\bezier{80}(51.00,2.00)(41.00,2.00)(41.00,12.00)
\multiput(51.00,2.00)(-1.17,0.08){2}{\line(-1,0){1.17}}
\multiput(48.66,2.16)(-0.51,0.12){4}{\line(-1,0){0.51}}
\multiput(46.63,2.63)(-0.25,0.11){7}{\line(-1,0){0.25}}
\multiput(44.91,3.41)(-0.14,0.11){10}{\line(-1,0){0.14}}
\multiput(43.50,4.50)(-0.11,0.14){10}{\line(0,1){0.14}}
\multiput(42.41,5.91)(-0.11,0.25){7}{\line(0,1){0.25}}
\multiput(41.63,7.63)(-0.12,0.51){4}{\line(0,1){0.51}}
\multiput(41.16,9.66)(-0.08,1.17){2}{\line(0,1){1.17}}
%\end
\put(51.00,22.00){\vector(-2,-1){9.00}} \put(51.00,22.00){\vector(2,-1){9.00}}
\put(60.00,7.00){\vector(0,1){10.00}} \put(60.00,7.00){\vector(-2,-1){9.00}}
\put(42.00,7.00){\vector(2,-1){9.00}} \put(42.00,7.00){\vector(0,1){9.67}}
\put(51.00,22.00){\vector(-2,-1){8.00}} \put(51.00,22.00){\vector(2,-1){8.00}}
\put(60.00,7.00){\vector(0,1){9.00}} \put(60.00,7.00){\vector(-2,-1){8.00}}
\put(42.00,7.00){\vector(2,-1){8.00}} \put(42.00,7.00){\vector(0,1){8.67}}
\put(51.00,2.33){\vector(0,1){19.33}} \put(51.00,2.33){\vector(0,1){18.33}}
\put(51.00,2.33){\vector(0,1){17.33}} \put(51.00,2.33){\vector(0,1){16.33}}
\put(60.00,16.67){\vector(-2,-1){17.67}}
\put(60.00,16.67){\vector(-2,-1){16.67}}
\put(60.00,16.67){\vector(-2,-1){15.67}}
\put(60.00,16.67){\vector(-2,-1){14.67}}
\put(42.33,16.67){\vector(2,-1){17.67}} \put(42.33,16.67){\vector(2,-1){16.67}}
\put(42.33,16.67){\vector(2,-1){15.67}} \put(42.33,16.67){\vector(2,-1){14.67}}
\put(2.23,7.00){\circle*{1.33}} \put(2.33,17.00){\circle*{1.33}}
\put(12.00,22.33){\circle*{1.33}} \put(19.67,17.00){\circle*{1.33}}
\put(19.67,7.00){\circle*{1.33}} \put(12.10,2.15){\circle*{1.33}}
\put(51.00,2.33){\circle*{1.33}} \put(60.10,7.50){\circle*{1.33}}
\put(60.00,16.50){\circle*{1.33}} \put(51.00,22.00){\circle*{1.33}}
\put(42.00,17.00){\circle*{1.33}} \put(42.00,7.67){\circle*{1.33}}
\end{picture}
\end{center}
\caption{Two generators of order 5 of the 1-cohomology group of the space of
plane immersed curves without 4-fold points} \label{hexagon}
\end{figure}
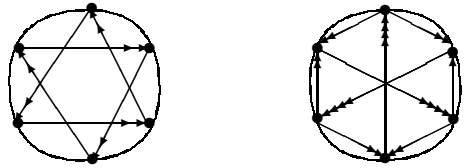

\medskip
Many invariants of immersions from Theorem 1 have elementary description.
Namely, the sum of three generators from Fig.~\ref{two}a is equal to the square
of the strangeness. Moreover, strangeness itself, the unique invariant of order
3 and the sum of two invariants of order 4 shown in Fig.~\ref{two}b, are
"index-type invariants," see \S~\ref{itype} below, thus initiating an infinite
series of finite-order invariants (one in each order) of this sort.
\medskip

{\sc Important Note.} Our notion of the order of invariants differs from the
one used in \cite{Arnold-20}---\cite{Arnold-21}, \cite{Shum-1}, \cite{Tabach}
etc. Any invariant of finite order $k$ in the sense of our work is also of
order $\le [k/2]$ in the sense of these works; the converse is false very much.

There are (among others) three equivalent definitions of finite order
invariants of knots in $\R^3$: 1) the "geometrical", in terms of resolved
discriminants and their filtrations, 2) the "axiomatic", in terms of finite
differences of knot diagrams, and 3) the "combinatorial" (developed in
\cite{PV}) in terms of homomorphisms of chord diagrams. The equivalence of two
first definitions was clear from the very beginning, their equivalence to the
third one is a nontrivial fact, conjectured by M.~Polyak and O.~Viro and proved
by M.~Goussarov.

There is a wide class of objects (including knots, ornaments, and doodles),
whose invariants can be calculated by the methods, developed in \cite{V-7},
\cite{V-11}, i.e. in the terms of resolved discriminants, thus leading to the
"geometrical" definition of finite-order invariants. An "axiomatic" elementary
reformulation of the resulting notion in our present situation also exists, but
it is not a straightforward translation of that from \cite{V-7}, see
\S~\ref{elemdef} below. I believe that it will lead to the most interesting
algebraic structures, reflecting the rich geometric structures staying behind
it.

The "combinatorial" definition and related aspects of the same invariants of
ornaments and doodles are introduced and investigated by A.~B.~Merkov,
\cite{Merx}---\cite{Merx-4} as a far generalization of the index-type
invariants from \cite{V-11}.

In particular, he proved that these invariants distinguish any two
nonequivalent collections of (arbitrarily many) plane curves without triple
intersections or selfintersections. However I believe that the techniques of
the present work allow to calculate all such invariants in the most direct and
regular way.
\medskip

I thank very much A.~B.~Merkov for numerous consultations and other multiform
help.

\section{Elementary theory}

This and the next sections are almost exact analogues of \S\S~1, 2 from
\cite{V-11}.

\subsection{First definitions and Reidemeister moves.}

\begin{definition}
A {\em doodle} is a $C^\infty$-map $\phi: S^1 \to \R^2$ such that for none
three different points $x,y,z \in S^1$ one the following conditions holds
\footnote{In a more general theory, see \cite{Khovanov}, \cite{Merx-4}, this
object is called an 1-doodle. We consider here only such one-component doodles
and call them simply {\em doodles}.}:
\begin{equation}
\label{main} \phi(x)=\phi(y)=\phi(z)
\end{equation}
\begin{equation}
\label{second} \phi'(x)=0, \phi(x) = \phi(y)
\end{equation}
\begin{equation}
\label{tert} \phi'(x)=\phi''(x)=0.
\end{equation}
An {\em I-doodle} (i.e. immersed doodle) is a doodle which is an immersion
(i.e. a map $\phi$ without degenerations of two types (\ref{main}) and
\begin{equation}
\label{imm} \phi'(x)=0 \ ).
\end{equation}

Two doodles (respectively, I-doodles) are {\em equivalent} if there is a
continuous family of doodles (I-doodles) connecting them. An {\em invariant} of
doodles or I-doodles is any function on the space of these objects, taking
equal values at equivalent objects.

A doodle is {\em regular} if it is an immersion having only transverse double
points.
\end{definition}

\prop \label{reidprop} Any equivalence class of doodles or I-doodles contains
regular doodles. Two regular doodles define equivalent doodles (respectively,
I-doodles) if and only if they can be transformed one into the other by a
finite sequence of isotopies of ${\bf R}^2$ (which do not change the
topological picture of the image of the doodle), and of local moves shown in
Fig.~\ref{reidem}a, b (respectively, \ref{reidem}a only). \eprop

\begin{figure}
\begin{center}
\unitlength 1.00mm \linethickness{0.4pt}
\begin{picture}(124.00,47.00)
%\emline(2.00,30.00)(27.00,30.00)
\put(2.00,30.00){\line(1,0){25.00}}
%\end
%\bezier{224}(27.00,22.00)(14.00,47.00)(2.00,22.00)
\multiput(27.00,22.00)(-0.11,0.21){11}{\line(0,1){0.21}}
\multiput(25.77,24.26)(-0.11,0.18){11}{\line(0,1){0.18}}
\multiput(24.54,26.30)(-0.11,0.16){11}{\line(0,1){0.16}}
\multiput(23.32,28.10)(-0.11,0.14){11}{\line(0,1){0.14}}
\multiput(22.10,29.69)(-0.11,0.12){11}{\line(0,1){0.12}}
\multiput(20.89,31.05)(-0.12,0.11){10}{\line(-1,0){0.12}}
\multiput(19.68,32.18)(-0.15,0.11){8}{\line(-1,0){0.15}}
\multiput(18.47,33.09)(-0.20,0.11){6}{\line(-1,0){0.20}}
\multiput(17.27,33.78)(-0.30,0.11){4}{\line(-1,0){0.30}}
\multiput(16.08,34.23)(-0.60,0.12){2}{\line(-1,0){0.60}}
\put(14.89,34.47){\line(-1,0){1.19}}
\multiput(13.70,34.48)(-0.59,-0.11){2}{\line(-1,0){0.59}}
\multiput(12.52,34.26)(-0.29,-0.11){4}{\line(-1,0){0.29}}
\multiput(11.34,33.82)(-0.20,-0.11){6}{\line(-1,0){0.20}}
\multiput(10.17,33.15)(-0.15,-0.11){8}{\line(-1,0){0.15}}
\multiput(9.00,32.26)(-0.12,-0.11){10}{\line(-1,0){0.12}}
\multiput(7.83,31.14)(-0.12,-0.13){10}{\line(0,-1){0.13}}
\multiput(6.68,29.80)(-0.12,-0.16){10}{\line(0,-1){0.16}}
\multiput(5.52,28.23)(-0.12,-0.18){10}{\line(0,-1){0.18}}
\multiput(4.37,26.43)(-0.11,-0.20){10}{\line(0,-1){0.20}}
\multiput(3.22,24.41)(-0.11,-0.22){11}{\line(0,-1){0.22}}
%\end
%\emline(2.00,14.00)(27.00,14.00)
\put(2.00,14.00){\line(1,0){25.00}}
%\end
%\bezier{140}(27.00,6.00)(14.00,18.00)(2.00,6.00)
\multiput(27.00,6.00)(-0.13,0.11){14}{\line(-1,0){0.13}}
\multiput(25.15,7.59)(-0.15,0.11){12}{\line(-1,0){0.15}}
\multiput(23.31,8.94)(-0.18,0.11){10}{\line(-1,0){0.18}}
\multiput(21.47,10.04)(-0.23,0.11){8}{\line(-1,0){0.23}}
\multiput(19.65,10.90)(-0.30,0.10){6}{\line(-1,0){0.30}}
\multiput(17.84,11.51)(-0.45,0.09){4}{\line(-1,0){0.45}}
\put(16.04,11.88){\line(-1,0){3.57}}
\multiput(12.47,11.88)(-0.44,-0.09){4}{\line(-1,0){0.44}}
\multiput(10.70,11.51)(-0.29,-0.10){6}{\line(-1,0){0.29}}
\multiput(8.94,10.90)(-0.22,-0.11){8}{\line(-1,0){0.22}}
\multiput(7.19,10.04)(-0.17,-0.11){10}{\line(-1,0){0.17}}
\multiput(5.45,8.94)(-0.14,-0.11){12}{\line(-1,0){0.14}}
\multiput(3.72,7.59)(-0.12,-0.11){14}{\line(-1,0){0.12}}
%\end
\put(14.00,2.00){\makebox(0,0)[cc]{{\large a}}}
\put(14.00,21.00){\makebox(0,0)[cc]{$\Updownarrow$}}
%\emline(37.00,11.00)(65.00,11.00)
\put(37.00,11.00){\line(1,0){28.00}}
%\end
\put(50.00,18.00){\makebox(0,0)[cc]{$\Updownarrow$}}
%\emline(82.00,34.00)(96.00,20.00)
\multiput(82.00,34.00)(0.12,-0.12){117}{\line(0,-1){0.12}}
%\end
%\emline(88.00,20.00)(102.00,34.00)
\multiput(88.00,20.00)(0.12,0.12){117}{\line(0,1){0.12}}
%\end
%\emline(103.00,27.00)(81.00,27.00)
\put(103.00,27.00){\line(-1,0){22.00}}
%\end
%\emline(88.00,15.00)(102.00,1.00)
\multiput(88.00,15.00)(0.12,-0.12){117}{\line(0,-1){0.12}}
%\end
%\emline(96.00,15.00)(82.00,1.00)
\multiput(96.00,15.00)(-0.12,-0.12){117}{\line(0,-1){0.12}}
%\end
%\emline(81.00,8.00)(103.00,8.00)
\put(81.00,8.00){\line(1,0){22.00}}
%\end
\put(92.00,2.00){\makebox(0,0)[cc]{{\large c}}}
\put(50.00,2.00){\makebox(0,0)[cc]{{\large b}}}
\put(92.00,18.00){\makebox(0,0)[cc]{$\Updownarrow$}}
%\emline(116.00,34.00)(124.00,26.00)
\multiput(116.00,34.00)(0.12,-0.12){67}{\line(0,-1){0.12}}
%\end
%\emline(116.00,26.00)(124.00,34.00)
\multiput(116.00,26.00)(0.12,0.12){67}{\line(0,1){0.12}}
%\end
%\emline(120.00,26.00)(120.00,34.00)
\put(120.00,26.00){\line(0,1){8.00}}
%\end
%\emline(116.00,32.00)(124.00,32.00)
\put(116.00,32.00){\line(1,0){8.00}}
%\end
%\emline(116.00,24.00)(124.00,16.00)
\multiput(116.00,24.00)(0.12,-0.12){67}{\line(0,-1){0.12}}
%\end
%\emline(116.00,16.00)(124.00,24.00)
\multiput(116.00,16.00)(0.12,0.12){67}{\line(0,1){0.12}}
%\end
%\emline(120.00,16.00)(120.00,24.00)
\put(120.00,16.00){\line(0,1){8.00}}
%\end
%\emline(116.00,20.00)(124.00,20.00)
\put(116.00,20.00){\line(1,0){8.00}}
%\end
%\emline(116.00,14.00)(124.00,6.00)
\multiput(116.00,14.00)(0.12,-0.12){67}{\line(0,-1){0.12}}
%\end
%\emline(116.00,6.00)(124.00,14.00)
\multiput(116.00,6.00)(0.12,0.12){67}{\line(0,1){0.12}}
%\end
%\emline(120.00,6.00)(120.00,14.00)
\put(120.00,6.00){\line(0,1){8.00}}
%\end
%\emline(116.00,8.00)(124.00,8.00)
\put(116.00,8.00){\line(1,0){8.00}}
%\end
\put(120.00,2.00){\makebox(0,0)[cc]{{\large d}}}
%\bezier{100}(37.00,22.00)(56.00,21.67)(55.33,28.00)
\put(37.00,22.00){\line(1,0){3.60}}
\multiput(40.60,22.00)(1.61,0.07){2}{\line(1,0){1.61}}
\multiput(43.81,22.13)(0.94,0.09){3}{\line(1,0){0.94}}
\multiput(46.63,22.40)(0.61,0.10){4}{\line(1,0){0.61}}
\multiput(49.05,22.80)(0.41,0.11){5}{\line(1,0){0.41}}
\multiput(51.08,23.33)(0.27,0.11){6}{\line(1,0){0.27}}
\multiput(52.72,24.00)(0.18,0.11){7}{\line(1,0){0.18}}
\multiput(53.96,24.80)(0.11,0.12){8}{\line(0,1){0.12}}
\multiput(54.81,25.73)(0.11,0.27){4}{\line(0,1){0.27}}
\put(55.27,26.80){\line(0,1){1.20}}
%\end
%\bezier{44}(55.33,28.00)(55.67,34.00)(51.00,34.33)
\put(55.33,28.00){\line(0,1){2.43}}
\multiput(55.23,30.43)(-0.10,0.31){6}{\line(0,1){0.31}}
\multiput(54.60,32.28)(-0.11,0.13){10}{\line(0,1){0.13}}
\multiput(53.46,33.55)(-0.35,0.11){7}{\line(-1,0){0.35}}
%\end
%\bezier{52}(51.00,34.33)(45.67,35.33)(46.00,28.00)
\put(51.00,34.33){\line(-1,0){1.84}}
\multiput(49.16,34.41)(-0.28,-0.11){5}{\line(-1,0){0.28}}
\multiput(47.74,33.87)(-0.11,-0.13){9}{\line(0,-1){0.13}}
\multiput(46.73,32.71)(-0.12,-0.35){5}{\line(0,-1){0.35}}
\multiput(46.15,30.94)(-0.07,-1.47){2}{\line(0,-1){1.47}}
%\end
%\bezier{104}(46.00,28.00)(46.00,21.00)(65.00,22.00)
\multiput(46.00,28.00)(0.09,-0.64){2}{\line(0,-1){0.64}}
\multiput(46.18,26.73)(0.11,-0.22){5}{\line(0,-1){0.22}}
\multiput(46.70,25.60)(0.11,-0.12){8}{\line(0,-1){0.12}}
\multiput(47.58,24.63)(0.18,-0.12){7}{\line(1,0){0.18}}
\multiput(48.81,23.80)(0.26,-0.11){6}{\line(1,0){0.26}}
\multiput(50.39,23.12)(0.39,-0.11){5}{\line(1,0){0.39}}
\multiput(52.32,22.59)(0.57,-0.10){4}{\line(1,0){0.57}}
\multiput(54.61,22.20)(1.32,-0.12){2}{\line(1,0){1.32}}
\put(57.24,21.96){\line(1,0){2.99}}
\multiput(60.23,21.88)(2.39,0.06){2}{\line(1,0){2.39}}
%\end
\end{picture}
\end{center}
\caption{Standard moves of quasidoodles} \label{reidem}
\end{figure}
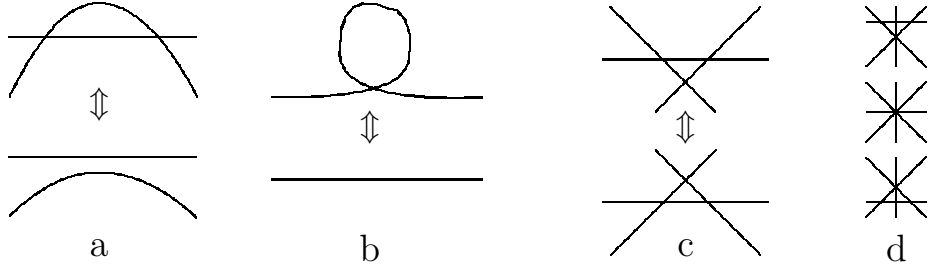

In other words, the move of Fig.~\ref{reidem}c is prohibited in the
classification of doodles, and both \ref{reidem}b, \ref{reidem}c in the case of
I-doodles.

The proof of this proposition is trivial.

\begin{definition} A {\em quasidoodle} is any $C^\infty$-map
$S^1 \to \R^2.$ The space of all such maps is denoted by $\K$. The space
$Imm(S^1,\R^2)$ of all immersions $S^1 \to \R^2$ will be denoted by $I\K$. The
{\em discriminant} (respectively, I-discriminant) $\Sigma \subset \K$
(respectively, $I\Sigma \subset I\K$) is the set of all maps from this space
for which one of prohibited conditions (\ref{main})---(\ref{tert})
(respectively, (\ref{main})) is satisfied.
\end{definition}

\prop The set $\Sigma$ is a closed subvariety of codimension 1 in $\K$. The set
of all maps, for which only (\ref{main}) is satisfied, is dense in $\Sigma,$
and the maps satisfying (\ref{second}) or (\ref{tert}) lie in its closure. The
set of quasidoodles having no degenerations of types (\ref{main}) and
(\ref{second}) but with allowed degenerations of type (\ref{tert}) is
path-connected in $\K.$ \eprop

The proof of the last assertion essentially coincides with that of the fact
that {\em all} embeddings $S^1 \to \R^3$ (maybe not regular) form a
path-connected subset in $C^\infty(S^1, \R^3),$ see e.g. \cite{CF}. All other
statements of the proposition are elementary.

\subsection{On Arnold's invariants of immersed plane curves.}

In \cite{Arnold-20} V.~I.~Arnold introduced three invariants of generic
immersed plane curves. One of them is the {\em strangeness}, defined as the
linking number in $I\K$ with the suitably (co)oriented variety $I\Sigma \subset
I\K.$ The coorientation of this variety, participating in this construction,
will be specified in \S~\ref{coorient}.

\subsection{Index-type invariants of I-doodles.}

\label{itype} Let us fix an orientation of the plane ${\bf R}^2.$

Recall that any closed oriented immersed curve $c$ in ${\bf R}^2$ defines an
integer-valued function $ind_c$ on its complement: for any point $t$ of the
complement, $ind_c(t)$ equals the (counterclockwise) rotation number of the
vector $(t,x)$ when $x$ runs one time along $c$.

Consider a regular I-doodle $\phi: S^1 \to \R^2$. To any self-intersection
point $x$ of the curve $c=\phi(S^1)$ assign its index $i(x)$ equal to the
arithmetical mean of four values of $ind_c$ in four neighboring components of
$\R^2 \sm c,$ see Fig.~\ref{ium}. Let us fix a regular (not intersection) point
$\ast $ in $c$ and define its index $i(\ast)$ as the greatest value of $ind_c$
in two neighboring domains of the complement of $c.$ For any selfintersection
point $x$ consider the frame in it, formed by (oriented) tangent vectors to
$c$. These vectors are ordered in correspondence with the number of visits of
$c$ after leaving the point $\ast$. Define the {\em sign} $\sigma(x)$ of the
point $x$ as the sign of the orientation of this ordered frame.

\begin{figure}
\begin{center}
\unitlength=1.00mm \special{em:linewidth 0.4pt} \linethickness{0.4pt}
\begin{picture}(16.00,16.00)
\put(15.00,9.00){\makebox(0,0)[cc]{$i$}}
\put(3.00,9.00){\makebox(0,0)[cc]{$i$}}
\put(9.00,16.00){\makebox(0,0)[cc]{$i+1$}}
\put(9.00,2.00){\makebox(0,0)[cc]{$i-1$}} \put(2.00,2.00){\vector(1,1){14.00}}
\put(2.00,16.00){\vector(1,-1){14.00}}
\put(5.00,13.00){\makebox(0,0)[cc]{{\large $\ast$}}}
\end{picture}
\end{center}
\caption{} \label{ium}
\end{figure}

For any integer $i$ and natural $\beta,$ denote by $\frac{i}{\beta}$ the number
$i(i-1) \cdots (i-\beta+1)/\beta!$, cf. \cite{Merx-2}, \cite{V-11}. It is easy
to see that this number is always integer.

For any natural $\beta,$ define the {\em $\beta$-th moment} $M(\beta)$ of the
regular doodle by the equality

\begin{equation}
\label{moment} M(\beta) = \sum_{x} {\sigma(x)} \binom{i(x)}{\beta} + 2
\binom{i(\ast)}{\beta+1} .
\end{equation}

\begin{proposition}[cf. \cite{Shum-3}, \cite{V-11}]
\label{mom} All numbers $M(\beta),$ $\beta=1,2, \ldots$, are invariants of
I-doodles, in particular do not depend on the choice of the distinguished point
$\ast.$ The first of them, $M(1),$ is the Arnold's strangeness.
\end{proposition}

\begin{remark} In the very similar case of ornaments such invariants were
introduced in \cite{V-11}, \S~1.4, as first nontrivial examples of finite-order
invariants, see also \cite{Merx-2}. In \cite{Tabach} similar expressions
appeared as invariants of one-component {\em long curves}, i.e. essentially of
curves with a fixed nonsingular point $*$. The formulae for all these
invariants contained only the terms similar to the first term of the right-hand
part of (\ref{moment}). Finally, A.~Shumakovich \cite{Shum-3} introduced a
correcting second term and obtained invariants independent on the choice of
this point: these his invariants coincide with (\ref{moment}) up to a linear
transformation with rational coefficients. (The simplest version of this
correcting term, corresponding to the case $\beta=1$ and providing a
combinatorial expression for the Arnold's strangeness, appeared previously in
\cite{Shum-1}.) Numerous more general combinatorial expressions for invariants
of doodles, ornaments, I-doodles, etc. were given in
\cite{Merx}--\cite{Merx-4}. It seems likely that the method described below
(see also \cite{V-11}, \cite{Merx-2}) is the most universal algorithm for
guessing such expressions: first one calculates several first elements of a
spectral sequence converging to the group of all invariants of finite degree,
and then finds an elementary interpretation for them; cf. also \cite{V-7},
\cite{V-94}.
\end{remark}

\subsection{Coorientation of the discriminant.}
\label{coorient}

The discriminant set $\Sigma$ has a natural (co)orientation in its regular
points: if we go along a generic path in the space $\K$ and traverse the
discriminant, doing the local surgery shown in Fig.~\ref{reidem}c, then there
is an invariant way to say, which one of these two resolved pictures lies on
the positive side of the discriminant, and which on the negative. There are
numerous equivalent definitions of this coorientation, see e.g.
\cite{Arnold-20}, \cite{Shum-1}, \cite{Merx-2}. One of them can be formulated
as follows: we consider the sum like (\ref{moment}) with $\beta=1,$ but with
summation only over 3 points participating in the surgery. The positive
(negative) side of discriminant is that with the greater (smaller) value of
this sum.

This coorientation is well defined even if the curve has forbidden multiple
(or, in the theory of I-doodles, forbidden singular) points far away from the
location of the surgery: in fact, this is a coorientation of the locally
irreducible branch of the discriminant set.

\section{Elementary definition of the order of invariants}
\label{elemdef}

\subsection{Degeneration modes and characteristic numbers.}

The orders of invariants of doodles and I-doodles will be defined in the same
way.

\begin{definition} Suppose that $j$ is a natural number, $j \ge 2.$
A {\it degree $j$ standard singularity\/} of doodles is a pair of the form \{a
quasidoodle $\phi:S^1 \to \R^2$; a point $x \in \R^2$\} such that
$\phi^{-1}(x)$ consists of exactly $j+1$ points $z_1, \ldots, z_{j+1}$, the map
$\phi$ close to all these points is an immersion, and the corresponding $j+1$
local branches of the curve $\phi(S^1)$ are pairwise nontangent at $x$. A
quasidoodle is called a {\it regular quasidoodle of complexity $i$,} if it is
an immersion, all its forbidden points (i.e. the points, at which at least
three different components meet) are standard singular points, and the sum of
degrees of these singularities is equal to $i$.
\end{definition}

Any regular quasidoodle can be obtained from regular doodles by a sequence of
elementary degenerations. Namely, first we move along a generic path in the
space $\K$, up to the first instant when some three points of $\phi(S^1)$ meet
at the same point, forming a regular singularity of degree 2 (we do not watch
the surgeries shown if Figs.~\ref{reidem}a and \ref{reidem}b). Then we consider
the vector subspace in $\K,$ consisting of maps gluing together these three
points of $S^1,$ and go along a generic path in it; at some instant either
another triple point occurs or a fourth branch joins these three. Again, we fix
the smaller subspace, consisting of maps gluing together all the same points,
and move inside it. On the third step a point of multiplicity 5 can occur, or
two points of multiplicities 4 and 3, of 3 points of multiplicity 3, etc. (In
this case we do not watch also the nonessential local moves like the one shown
in Fig.~\ref{noness}.) At the last step we get our quasidoodle.

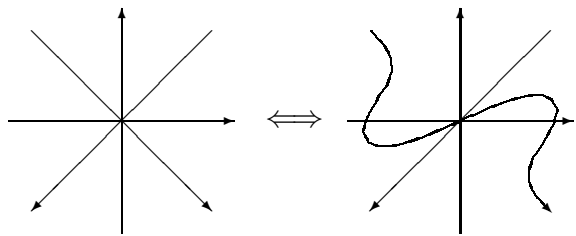
\begin{figure}
\begin{center}
\unitlength 1.00mm \linethickness{0.4pt}
\begin{picture}(75.00,30.00)
\put(0.00,15.00){\vector(1,0){30.00}} \put(15.00,0.00){\vector(0,1){30.00}}
\put(27.00,27.00){\vector(-1,-1){24.00}} \put(3.00,27.00){\vector(1,-1){24.00}}
\put(38.00,15.00){\makebox(0,0)[cc]{$\Longleftrightarrow$}}
\put(45.00,15.00){\vector(1,0){30.00}} \put(60.00,0.00){\vector(0,1){30.00}}
\put(72.00,27.00){\vector(-1,-1){24.00}}
%\bezier{36}(49.00,18.00)(45.67,13.00)(48.00,12.00)
\multiput(49.00,18.00)(-0.12,-0.21){12}{\line(0,-1){0.21}}
\multiput(47.59,15.53)(-0.11,-0.37){5}{\line(0,-1){0.37}}
\multiput(47.05,13.68)(0.12,-0.21){8}{\line(0,-1){0.21}}
%\end
%\bezier{56}(48.00,12.00)(50.00,10.33)(60.00,15.00)
\multiput(48.00,12.00)(0.24,-0.10){4}{\line(1,0){0.24}}
\put(48.97,11.61){\line(1,0){1.48}}
\multiput(50.45,11.62)(0.50,0.10){4}{\line(1,0){0.50}}
\multiput(52.44,12.03)(0.36,0.12){7}{\line(1,0){0.36}}
\multiput(54.94,12.85)(0.28,0.12){18}{\line(1,0){0.28}}
%\end
\put(72.00,3.00){\vector(1,-1){0.20}}
%\bezier{36}(71.00,12.00)(74.33,17.00)(72.00,18.00)
\multiput(71.00,12.00)(0.12,0.21){12}{\line(0,1){0.21}}
\multiput(72.41,14.47)(0.11,0.37){5}{\line(0,1){0.37}}
\multiput(72.95,16.32)(-0.12,0.21){8}{\line(0,1){0.21}}
%\end
%\bezier{56}(72.00,18.00)(70.00,19.67)(60.00,15.00)
\multiput(72.00,18.00)(-0.24,0.10){4}{\line(-1,0){0.24}}
\put(71.03,18.39){\line(-1,0){1.48}}
\multiput(69.55,18.38)(-0.50,-0.10){4}{\line(-1,0){0.50}}
\multiput(67.56,17.97)(-0.36,-0.12){7}{\line(-1,0){0.36}}
\multiput(65.06,17.15)(-0.28,-0.12){18}{\line(-1,0){0.28}}
%\end
%\bezier{52}(71.00,12.00)(66.67,7.50)(72.00,3.00)
\multiput(71.00,12.00)(-0.12,-0.16){11}{\line(0,-1){0.16}}
\multiput(69.69,10.27)(-0.12,-0.35){5}{\line(0,-1){0.35}}
\multiput(69.10,8.54)(0.06,-0.87){2}{\line(0,-1){0.87}}
\multiput(69.22,6.81)(0.12,-0.25){7}{\line(0,-1){0.25}}
\multiput(70.05,5.08)(0.11,-0.12){17}{\line(0,-1){0.12}}
%\end
%\bezier{52}(49.00,18.00)(53.33,22.50)(48.00,27.00)
\multiput(49.00,18.00)(0.12,0.16){11}{\line(0,1){0.16}}
\multiput(50.31,19.73)(0.12,0.35){5}{\line(0,1){0.35}}
\multiput(50.90,21.46)(-0.06,0.87){2}{\line(0,1){0.87}}
\multiput(50.78,23.19)(-0.12,0.25){7}{\line(0,1){0.25}}
\multiput(49.95,24.92)(-0.11,0.12){17}{\line(0,1){0.12}}
%\end
\end{picture}
\end{center}
\caption{Nonessential move of a complicated multiple point} \label{noness}
\end{figure}

Any such sequence of paths is called the {\em degeneration process} of our
regular quasidoodle.
\medskip

\noindent {\bf Examples.} If our quasidoodle has only one triple point, then it
can be obtained by two essentially different processes, corresponding to two
its resolutions shown in Fig.~\ref{reidem}c.

A quasidoodle with one singular point of multiplicity 4 has $16= 4\times 2
\times 2$ essentially different degeneration processes: at the first step any 3
of 4 points can meet at the same point of $\R^2$ (and this can be done in two
different ways, see Fig.~\ref{reidem}c), and on the second the fourth point
joins them in one of two different ways shown in Fig.~\ref{reidem}d.

If our quasidoodle has one more point of multiplicity 3, then there are $96 = 4
\times 3\times 2^3$ different degeneration processes: the points of the second
group can meet before, after, or between of two steps of degeneration of the
first group.

Any degeneration process of a regular quasidoodle and
 any invariant of doodles (or I-doodles)
defines a {\it characteristic number}, cf. \S~2 of \cite{V-11}. Namely, if we
have a quasidoodle of some complexity $j$ and a degeneration process $DP$ of
it, then there are exactly two quasidoodles of smaller complexity, whose
degeneration processes coincide with $DP$ without its last step. If at this
last step some new group of multiplicity 3 occurs, then these are two
resolutions of this group shown in Fig.~\ref{reidem}c; if at this step one
branch of $\phi(S^1)$ joined an existing group of multiplicity $\ge 3,$ then we
can move this branch to exactly two sides from this point, see
Fig.~\ref{reidem}d.

In all cases these two resolutions are ordered, i.e. one of them can be called
positive and the other negative. In the first case this order is described in
\S~\ref{coorient}.

In the second we define the index of a multiple point as the arithmetical mean
of indices of points from all neighboring components of the complement of
$\phi(S^1)$, and call {\em positive} the side for which the close point of
multiplicity $\ge 3$ has greater index with respect to the curve.

The characteristic number, which an invariant defines at the pair \{a regular
quasidoodle, some its degeneration process\} is equal to the difference of
similar numbers at two corresponding one-step resolutions (at the positive
minus at the negative one) supplied with the same degeneration process less its
last step. (For these easier singularities these characteristic numbers are
well defined by the inductive conjecture.)

\begin{definition}
An invariant is {\em of order} $j$ if for any quasidoodle, whose complexity is
greater than $j$, and any its degeneration process the corresponding
characteristic number is equal to 0.
\end{definition}

\begin{proposition}[cf. Theorem 4 in \cite{V-11}]
\label{ordofind} Any index-type invariant $M(\beta)$ from \S~\ref{itype} is of
order $\beta+1$. \quad $\Box$
\end{proposition}

A wide class of invariants, generalizing these from \S~\ref{itype}, was
constructed by A.~B.~Merkov, see \cite{Merx-2}. A further generalization of
these invariants, \cite{Merx-4}, classifies all doodles up to equivalence; in
particular it is as strong as entire space of all finite-order invariants.

\subsection{Coding and calculation of finite-order invariants.}
\label{codcal}

Of course, the characteristic numbers of a degeneration process (and hence also
the notion of the order) depend not of its geometrical realization, but only of
some discrete data related with it, such as the combinatorial type of the set
of points in $S^1$ pasted together at different its steps. Let us describe
these data.

Let $A$ be a finite series of integer numbers $A = (a_1 , a_2, \ldots , a_m)$,
all of which are $\ge 3.$ Denote by $|A|$ the number $a_1 + \ldots + a_m$, and
by $\#A$ the number of elements $a_l$ of the series $A$ (denoted in the
previous line by $m$).

\begin{definition}
An $A$-{\it configuration} is a collection of $|A|$ pairwise different points
in $S^1$ divided into groups of cardinalities $a_1, \ldots, a_{\#A}$. Two
$A$-configurations are {\it equivalent} if they can be transformed one into the
other by an orientation-preserving diffeomorphism of $S^1$. A quasidoodle
$\phi:S^1 \to {\bf R}^2$ {\it respects} an $A$-configuration if it sends any of
corresponding $\#A$ groups of points into one point in ${\bf R}^2$. $\phi$ {\it
strictly respects} the $A$-configuration if, moreover, all these $\#A$ points
in ${\bf R}^2$ are distinct, have no extra preimages than these $|A|$ points,
and $\phi$ has no extra points in ${\bf R}^2$ at which images of three of more
different points of $S^1$ meet.
\end{definition}

Obviously, the space of all quasidoodles respecting a given
$A$-con\-fi\-gu\-ra\-tion $J$ is a linear subspace of codimension $2(|A|-\#A)$
in the space $\K$ of all quasidoodles. The number $|A|-\#A$ is thus called the
{\em complexity} of the configuration. We shall denote this subspace by
$\chi(J)$. The set of all quasidoodles, which {\em strictly} respect this
configuration, is an open dense subset in this subspace.

\begin{definition}
A {\em degeneration mode} of an $A$-configuration is some arbitrary order of
marking all the points of the configuration, satisfying the following
conditions: on any step we mark either some three points of some of $\#A$
groups (if none point of the same group is already marked) or one point of a
group, some three or more points of which are already marked.
\end{definition}

Any degeneration process of a regular quasidoodle $\phi$ defines in the obvious
way a degeneration mode of the $A$-configuration strictly respected by $\phi.$

Let $M$ be an invariant of doodles, and $J$ an $A$-configuration.

\begin{proposition}[cf. Theorem 5 in \cite{V-11}]
\label{symbol} If $M$ is an invariant of order $i$, and $|A| - \#A = i,$ then
for any regular quasidoodle $\phi,$ which strictly respects the
$A$-con\-fi\-gu\-ra\-ti\-on $J$, any characteristic number defined by the
triple consisting of $M$, $\phi$ and a degeneration process of $\phi$, depends
only on the pair consisting of the configuration $J$ and its degeneration mode
defined by this degeneration process. \quad $\Box$
\end{proposition}

\noindent {\bf Corollary.} {\em For any $j$ the group of order $j$ invariants
of doodles or I-doodles is finitely generated.}
\medskip

Indeed, the number of its generators does not exceed the sum (over all
equivalence classes of $A$-configurations of complexity $\le i$) of numbers of
their degeneration modes. \quad $\Box$
\medskip

Any invariant of order $i$ can be encoded by its {\it characteristic table}
which we now describe.

This table has $i+1$ levels numbered by $0, 1, \ldots, i$. The $l$-th level
consists of several cells, which are in one-to-one correspondence with all
possible pairs consisting of

a) an equivalence class of $A$-configurations of complexity $l$ in $S^1$,

b) a degeneration mode of this $A$-configuration.

In each cell we indicate

a) a picture (or a code) representing a ``model'' regular quasidoodle, which
strictly respects some $A$-confi\-gu\-ra\-ti\-on from the corresponding
equivalence class (this picture is the same for all invariants),

b) a degeneration process of this quasidoodle, defining this degeneration mode
(also not depending on the invariant), and

c) the characteristic number, which our invariant and the degeneration process,
corresponding to the cell, assign to this quasidoodle.

By the Proposition~\ref{symbol}, we may not specify the pictures and
degeneration processes in the cells of the highest ($i$-th) level of the table:
indeed, the corresponding characteristic numbers depend only on the data
indexing the cell.

For instance, the $0$-th level consists of the trivial $\bigcirc$-like doodle,
and the corresponding characteristic number equals $0$ (we can normalize all
invariants so that they take zero value on the trivial doodle). The $1$-st
level is empty, because there are no configurations of complexity $1$.
\medskip

Having these data, we can calculate our invariant by the inductive process,
coinciding identically with the one described in \S~3 of \cite{V-11} or \S~4.2
of \cite{V-7} (where, however, the characteristic numbers sometimes are called
"actuality indices", and all (quasi)doodles should be replaced by
(quasi)ornaments or (singular) knots).

\begin{remark}
Of course, the characteristic numbers corresponding to different degeneration
modes of the same regular quasidoodle satisfy some natural relations. For
instance, if two degeneration modes differ only by a reordering of markings,
preserving their order inside any group of the $A$-configuration, then the
corresponding characteristic numbers coincide.

Less trivial identities, relating different degeneration modes inside the same
group, follow from the differentials in the chain complex of {\em connected
hypergraphs}, see \cite{V-10}, \cite{BW}.
\end{remark}

\section{Invariants of doodles in terms of the resolved discriminant}

We shall work with the space $\K\equiv C^\infty(S^1,\R^2)$ as with an Euclidean
space of a very large but finite dimension $\Delta$. The justification of this
assumption uses the finite-dimensional approximations of this space and is
similar to that given in \cite{V-11}, \cite{V-7}. A rigorous reader can
everywhere below consider $\K$ as a generic finitedimensional subspace in
$C^\infty(S^1, \R^2).$
\medskip

In particular, we shall use the Alexander duality formula
\begin{equation}
\label{alex} \tilde H^i(\K \setminus \Sigma) \simeq \bar
H_{\Delta-1-i}(\Sigma),
\end{equation}
where $\tilde H^*$ is the usual reduced cohomology group (we are especially
interested in the group $\tilde H^0(\K \sm \Sigma)$ of invariants taking zero
value on the trivial doodle), and $\bar H_*$ is the {\em Borel--Moore homology
group}, i.e. the homology group of the one-point compactification reduced
modulo the added point.

\subsection{Simplicial resolution of the discriminant variety.}

Denote by $\Psi$ the {\em configuration space} of all unordered collections of
three points in $S^1,$ so that $\Psi=(S^1)^3/S(3).$ It is a smooth
3-dimensional manifold with corners, homotopy equivalent to $S^1$. More
precisely, it is the space of an orientable fiber bundle over $S^1,$ whose
projection $p$ sends a triple of points to their sum in the Lie group $S^1,$
the fiber is a closed filled triangle, and the monodromy over the entire base
$S^1$ provides the cyclic permutation of sides and vertices of the triangles.
The "zero section", consisting of centers of these fibers, consists of
configurations, all whose points are at the distance $2\pi/3$ one from the
other.

Let us fix a space $\R^N$ of a huge dimension (much greater than that of the
space $\K$) and fix a generic embedding $\lambda: \Psi \to \R^N.$ For any point
$\phi \in \Sigma$ consider all such points $\{x,y,z\} \in \Psi$ that one of
three conditions holds:

a) $x \ne y \ne z \ne x$ and $\phi(x)=\phi(y)=\phi(z)$;

b) $x=z \ne y$ and $\phi'(x)=0, \phi(x)=\phi(y)$;

c) $x=y=z$ and $\phi'(x)=\phi''(x)=0.$

Then consider all points $\lambda(\{x,y,z\}) \in \R^N$ for all such triples
$\{x,y,z\}$. Since our embedding is generic (and $N$ is sufficiently large)
then for any $\phi \in \Sigma$ the convex hull of all such points is a simplex
with vertices at all these points. (Using the {\em generic} finite-dimensional
approximations of the space $\K$ we can ignore the situation when the number of
such triples is infinite, moreover, we can assume that the number of such
triples has a finite upper estimate, uniform over all $\phi \in \Sigma.$)

Denote this simplex by $\sigma(\phi).$

Finally, define the {\em resolution set} $\sigma \subset \K \times \R^N$ as the
union of all simplices of the form $\phi \times \sigma(\phi)$ over all $\phi
\in \Sigma$.
\medskip

The obvious projection $\K \times \R^N \to \K$ provides the map $\pi: \sigma
\to \Sigma.$ By definition, this map is surjective, and by the previous
"finiteness assumption" it is also proper.

\begin{proposition}[cf. \cite{V-7}, \cite{V-11}]
\label{homeq} The map $\pi$ provides the homotopy equivalence of one-point
compactifications of spaces $\sigma$ and $\Sigma$. \quad $\Box$
\end{proposition}

In particular, the Borel--Moore homology groups (see (\ref{alex})) of these
spaces are canonically isomorphic.

\subsection{$A$-cliques and the main filtration of the resolved discriminant.}
\label{mainfil}

The space $\sigma$ admits a natural filtration, which can be defined in two
equivalent ways. To do it, we need to extend the notion of an $A$-configuration
used in \S~\ref{codcal}. Again, let $A$ be a finite series of $\#A$ integer
numbers $A = (a_1 , a_2, \ldots , a_{\#A})$, all of which are $\ge 3.$

\begin{definition} An $A$-{\em clique} is an unordered collection of
$a_1 + \cdots + a_{\#A}$ points in $S^1,$ divided into groups of cardinalities
$a_1, \ldots, a_{\#A},$ such that a) points of different groups do not coincide
geometrically; b) points inside a group can coincide, but with multiplicity at
most 3. Again, the {\em complexity} of a clique $J$ is the number $|A|-\#A$;
another important characteristic, $\rho(J)$, is the number of geometrically
distinct points in it, i.e. the dimension of the space of cliques equivalent to
it.

The map $\phi:S^1 \to \R^2$ {\em respects} an $A$-clique if it glues together
all geometrically distinct points inside any its group, satisfies the condition
$\phi'=0$ at all points of multiplicity 2, and satisfies the condition $\phi'=
\phi''= 0$ at all points of multiplicity 3. $\phi$ {\em strictly respects} it,
if additionally it does not respect any cliques of larger complexity.
\end{definition}

For any $A$-clique $J$, consider all triples of points in $S^1$ belonging to
the same its group. All such triples are the points of the configuration space
$\Psi.$ Consider the images in $\R^N$ of all these points under the embedding
$\lambda$ and define the simplex $\sigma(J) \subset \K \times \R^N $ as the
convex hull of all such points.

\begin{example} Suppose that $\#A=1,$ $a_1=4,$ and the unique group
of our $A$-clique is a quadruple of points $(x,x,y,z),$ exactly two of which
coincide. Then the simplex $\sigma(J)$ is a triangle with 3 vertices $(x,x,y)$,
$(x,x,z)$, $(x,y,z)$.
\end{example}

Now, for any natural $i$ we take all quasidoodles, {\em strictly} respecting
all possible $A$-cliques of complexities $\le i,$ then consider the union of
their complete preimages in $\sigma$ and, finally, define the term $\sigma_i$
of our {\em main filtration} as the closure of this union. It contains also
some points of the form $\phi \times \theta \in \K \times \R^N,$ where $\phi$
is a quasidoodle of complexity $ >i$, and $\theta$ is some boundary point of
the corresponding simplex $\sigma(\phi).$

Equivalently, for any $A$-clique $J$ we can define the linear subspace $\chi(J)
\subset \K$ consisting of all quasidoodles respecting (strictly or not) this
clique. The term $\sigma_i \subset \sigma$ of the main filtration is then
defined as the union of all subsets $\chi(J) \times \sigma(J) \subset \K \times
\R^N$ over all cliques $J$ of complexity $\le i$. It is easy to see that these
two definitions of the main filtration are equivalent.

\begin{definition} An element of the group
$\bar H_*(\Sigma) \equiv \bar H_*(\sigma)$ is {\em of order} $i$ if it can be
realized by a locally finite cycle lying in the term $\sigma_i$ of this
filtration. In particular, an invariant of doodles is of order $i$ if its class
in the group (\ref{alex}) can be realized as a linking number with the direct
image of a cycle lying in $\sigma_i.$
\end{definition}

\begin{proposition}[cf. \cite{V-11}, Theorem 7]
\label{equival} The last definition of the order of invariants is equivalent to
that given in \S~\ref{elemdef}. \quad $\Box$
\end{proposition}

\section{Calculation of invariants of doodles}

Consider the spectral sequence $E^r_{p,q}$ calculating the Borel--Moore
homology group of the space $\sigma$ and generated by our filtration. Its term
$E^1_{p,q}$ is isomorphic to $\bar H_{p+q}(\sigma_p \sm \sigma_{p-1})$.

\begin{proposition}[cf. \cite{V-11}]
\label{dimres} All groups $E^1_{p,q}$ with $p+q \ge \Delta$ are equal to 0.
\end{proposition}

The proof of this proposition will be given in \S~\ref{revbl}.
\medskip

Hence, for the calculation of the group of invariants (by (\ref{alex})
coinciding with $\bar H_{\Delta-1}(\sigma)$) only the $(\Delta-1)-$ and
$(\Delta-2)$-dimensional Borel--Moore homology groups of these spaces $\sigma_i
\sm \sigma_{i-1}$ are interesting.

In this section we calculate these groups for $i \le 4.$

\subsection{Stratification of the resolved discriminant.}
\label{primstrat}

By construction, any space $\sigma_i \sm \sigma_{i-1}$ consists of several
$\J$-{\em blocks}, numbered by all equivalence classes $\J$ of $A$-cliques of
complexity exactly $i$. Given such an equivalence class $\J,$ the corresponding
block $B(\J)$ is the space of a fiber bundle, whose base is the space of all
$A$-cliques $J$ of this class, and the fiber over a clique $J$ is the direct
product of

a) a linear subspace of codimension $2i$ in $\K$, consisting of all
quasidoodles respecting $J$ (the vector bundle of such subspaces is always
orientable) and

b) a dense subset of the simplex $\sigma(J)$ (namely, this simplex minus some
its faces, which may belong to $\sigma_{i-1}$\footnote{If $J$ is an
$A$-configuration, i.e. has no multiple points, then these are exactly the
faces corresponding to {\em not connected} 3-hypergraphs, see \cite{V-10},
\cite{BW}}).
\medskip

By the Thom isomorphism, the Borel--Moore homology group $\bar H_*$ of any such
block is canonically isomorphic to the group $\bar H_{*-(\Delta-2i)}$ of the
space of only the second bundle of complexes b). Such spaces will be called the
{\em reduced} $\J$-blocks.

The bases of these fiber bundles are $\rho(\J)$-dimensional manifolds, where
$\rho(\J)$ is the number of geometrically distinct points in any clique $J$ of
the class $\J.$

\subsection{The auxiliary filtration.}

\begin{definition} We introduce the {\em auxiliary filtration}
in the space $\sigma_i \sm \sigma_{i-1}$, defining its term $F_\alpha$ as the
union of all above-described blocks over all classes $\J$ of $A$-cliques such
that $\rho(\J) \le \alpha.$
\end{definition}

\begin{example} The term $\sigma_2$ of the main filtration
consists of exactly 3 terms $F_1 \subset F_2 \subset F_3 \equiv \sigma_2$ of
the auxiliary filtration, because the $(3)$-cliques $\{x,y,z\}$ can consist of
1, 2 or 3 geometrically different points.
\end{example}

\begin{definition} The spectral sequence, calculating the group
$\bar H_*(\sigma_i \sm \sigma_{i-1})$ and generated by this auxiliary
filtration, is called the {\em auxiliary spectral sequence} in contrast to the
{\em main} one generated by the main filtration and calculating the homology
groups of entire $\sigma$.
\end{definition}

\subsection{Columns $p=1$ and $p=2$ of the main spectral sequence.}

The term $\sigma_1$ is empty, because there are no cliques of complexity 1.

\prop \label{si2} The colum $\{p=2\}$ of the term $E^1$ of the main spectral
sequence has only the following nontrivial elements: $E^1_{2,\Delta-5} \simeq
\Z$ and $E^1_{2,\Delta-6} \simeq \Z.$ \eprop

\begin{proof}
The term $\sigma_2$ is the space of a fiber bundle, whose base is the
configuration space $\Psi,$ and the fiber over its point $\{x,y,z\}$ is the
linear subspace of codimension 4 in $\K$ consisting of all quasidoodles
respecting the corresponding $(3)$-clique. This bundle is orientable, therefore
$\bar H_*(\sigma_2) = \bar H_{*-\Delta+4}(\Psi) \simeq H_{*-\Delta+4}(S^1).$
\end{proof}

\subsection{Revised resolution.}
\label{revbl}

Before continuing, we improve slightly the construction of our $\J$-blocks,
described in the subsection \ref{primstrat}. Namely, any of these blocks
contains a deformation retract having the same Borel--Moore homology group; it
will be convenient to consider these new blocks $VB(\J)$, which we (in
accordance with the terminology of \cite{V-11}) shall call the {\em visible
blocks.}

Again, any such block, corresponding to an equivalence class $\J$ of
$A$-cliques, is the space of a fibered product of two bundles, whose base and
the first factor of the fiber are the same as previously (i.e. respectively the
space of all cliques $J \in \J$ and an oriented vector space of dimension
$\Delta-2(|A|-\#A)$). However, the second factors of the fibers of this new
bundle are some subcomplexes of the barycentric subdivision of the
corresponding fibers of the former bundle of simplices.

Namely, consider any $A$-clique $J$ and the corresponding simplex $\sigma(J)$.
Any vertex of this simplex, i.e. a triple of points of our $A$-clique lying
inside one its group, defines a vector subspace of codimension $4$ in $\K$.
Consider all these subspaces corresponding to our clique. (E.g., if the clique
is an $A$-configuration, then there are exactly $\sum_{i=1}^{\#A}
\binom{a_i}{3}$ such different subspaces.) These subspaces together with all
their possible intersections form a partially ordered set $\Pi(J)$ with respect
to the relation of the (inverse) inclusion. This poset has unique maximal
element: the subspace $\chi(J)$, i.e. the intersection of all our subspaces.

Having such a partially ordered set, we can define its {\em order complex}
$\Diamond(J)$ (see e.g. \cite{B}): this is a formal simplicial complex, whose
simplices are all the strictly monotone sequences of elements of our poset.

E.g., if the $A$-clique $J$ is an $A$-configuration, then such simplices of the
maximal dimension in $\Diamond (J)$ are nothing but the degeneration modes of
$J$ described in \S~\ref{codcal}.

This order complex $\Diamond(J)$ can be naturally considered as a subcomplex of
the barycentric subdivision of the simplex $\sigma(J)$ (while its maximal
element $\{\chi(J)\}$ corresponds to the center of $\sigma(J)$).

If the complexity $|A|-\#A$ of $J$ is equal to $j$, then the set $\Diamond(J)
\cap \sigma_{j-1}$ of "marginal" faces of this complex consists of all its
faces not containing the main vertex $\{\chi(J)\}$.

Define the {\em visible block} $VB(\J)$ corresponding to our equivalence class
$\J$ of $A$-cliques as the subset of the previously defined block $B(\J)$,
described in \S~\ref{primstrat}, in which elements of the fibers $\sigma(J)$
should belong to the subcomplex $\Diamond(J) \subset \sigma(J).$

Define the revised resolved discriminant $\Diamond \subset \sigma$ as the union
of all such revised blocks $VB(\J)$. It has a natural {\em main} filtration
$\{\Diamond_i\}$ induced by the identical embedding from the main filtration in
$\sigma$. Similarly, in any space $\Diamond_i \sm \Diamond_{i-1} $ the
auxiliary filtration is induced from that in $\sigma_i \sm \sigma_{i-1}.$

\begin{proposition}[see \cite{V-11}, \cite{Phasis}]
\label{homeq2} The inclusion $\Diamond \hookrightarrow \sigma$ induces a
homotopy equivalence of one-point compactifications of these spaces. The same
is true for the inclusion $\Diamond_j \hookrightarrow \sigma_j$ of any terms of
the main filtrations and also for the inclusions $VB(\J) \hookrightarrow B(\J)$
of blocks in them corresponding to the same classes of equivalent $A$-cliques
of complexity $j$. In particular this inclusion induces an isomorphism of both
the main and auxiliary spectral sequences calculating the Borel--Moore homology
groups of all these objects. \quad $\Box$
\end{proposition}

The Proposition \ref{dimres} follows immediately from this one and the
following lemma.
\medskip

\begin{lemma} The dimension of any block $VB(\J)$
of $\Diamond$ does not exceed $\Delta-1.$
\end{lemma}

\begin{proof}
This dimension consists of three numbers:

a) the dimension $\rho(\J) $ of the space of cliques of our class (which does
not exceed the number $|A|$);

b) the dimension $\Delta-2(|A|-\#A)$ of the standard fiber $\chi(J)$ of the
first (vector) bundle;

c) the dimension of the order complex of subspaces associated with the clique
$J$.

The last dimension is equal to the length of the maximal monotone chain of
subspaces constituting our poset $\Pi(J)$ minus 1. This length is equal to
$|A|-2\#A$. Indeed, all our subspaces are of even codimension in $\K$, their
maximal element is the space $\chi(J)$ of codimension $2(|A|-\#A)$, and there
are exactly $\#A$ steps in the chain, when the dimension jumps by 4 (when we
start a new group of points; the very first step is one of them).

Finally, we obtain that the dimension of our block is not greater than $|A| +
\Delta-2(|A|-\#A)+ |A|-2\#A-1.$
\end{proof}

\noindent {\bf Corollary.} Calculating the invariants of doodles, we can
consider only the $\J$-blocks corresponding to such $A$-cliques $J$, that
$|A|-\rho(A) \le 1,$ i.e. either all their points are geometrically distinct or
there is at most one point of multiplicity 2. Moreover, the blocks of the
latter kind can provide only relations in the group of invariants $\bar
H_{\Delta-1}(\sigma)$, but not its generators.

\subsection{The third column of the spectral sequence.}

\thm \label{s3} The group $E^1_{3,q} = \bar H_{3+q}(\Diamond_3 \sm \Diamond_2)$
of the main spectral sequence is trivial for all numbers $q$ not equal to
$\Delta-6$ and $\Delta-7$, while $E^1_{3,\Delta-6} \simeq \Z \simeq
E^1_{3,\Delta-7}.$ \etheorem

The proof of this theorem occupies the rest of this subsection.
\medskip

The term $\Diamond_3 \sm \Diamond_2$ consists of exactly 4 $\J$-blocks
corresponding to different equivalence classes $\J$ of cliques of complexity 3:
the main block A (all 4 points in the cliques are distinct), the block B of
auxiliary filtration 3 (exactly two points coincide), and two blocks of
auxiliary filtration 2: C (one simple point and one point of multiplicity 3)
and D (two double points).

Consider the corresponding reduced blocks (see \S~\ref{primstrat}) $[A], [B],
[C]$ and $[D].$
\medskip

The main reduced block $[A]$ is the space of a fiber bundle, whose base $\J$ is
the configuration space $B(S^1,4)$ of subsets of cardinality 4 in $S^1,$ and
the fiber is a cross with its four endpoints removed. The vertices of such a
cross correspond to all possible choices of some 3 points of the 4-point
configuration; this notation is transparent in Fig.~\ref{str3dif}a.

The base $B(S^1,4),$ in its turn, is the space of the fiber bundle
\begin{equation}
\label{projconf} p: B(S^1,4) \to S^1
\end{equation}
where the projection sends a quadruple of points in $S^1 = \R/\Z$ to their sum
(mod $\Z$), and the fiber is diffeomorphic to an open 3-dimensional disc. The
monodromy over the basis circle of (\ref{projconf}) violates the orientation of
the bundle of 3-dimensional discs and acts on the bundle of crosses as a cyclic
permutation of their edges.

Thus the Wang exact sequence of this bundle gives us the following assertion.

\prop \label{mstrat3} The Borel--Moore homology group of the reduced main block
$[A]$
 of $\Diamond_3\sm \Diamond_2$
is nontrivial only in dimensions $5$ and $4$ and is isomorphic to $\Z$ in these
dimensions. The generator of the $5$-dimensional group is swept out by the
bundle over $B(S^1,4)$ of 1-chains shown in Fig.~\ref{one}b $\equiv$
Fig.~\ref{str3dif}a (i.e. the unique chains antiinvariant under the monodromy
action). The generator of the $4$-dimensional group is swept out by the fiber
bundle, whose base is the fiber $p^{-1}(0)$ of the bundle (\ref{projconf}), and
fibers are 1-chains shown in these pictures by the sum of two right-hand
arrows. \quad $\Box$ \eprop

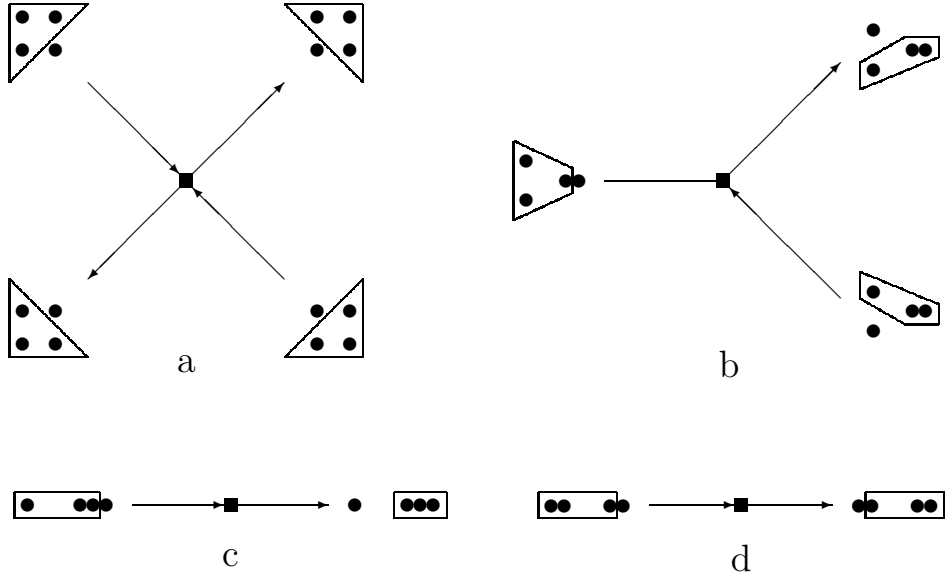
\begin{figure}
\begin{center}
\unitlength 0.87mm \linethickness{0.4pt}
\begin{picture}(143.00,57.00)
\put(3.00,5.00){\circle*{2.00}} \put(8.00,5.00){\circle*{2.00}}
\put(8.00,10.00){\circle*{2.00}} \put(3.00,10.00){\circle*{2.00}}
\put(27.00,29.00){\rule{2.00\unitlength}{2.00\unitlength}}
\put(43.00,15.00){\vector(-1,1){14.00}}
\put(27.00,29.00){\vector(-1,-1){14.00}}
\put(13.00,45.00){\vector(1,-1){14.00}} \put(29.00,31.00){\vector(1,1){14.00}}
\put(48.00,5.00){\circle*{2.00}} \put(53.00,5.00){\circle*{2.00}}
\put(53.00,10.00){\circle*{2.00}} \put(48.00,10.00){\circle*{2.00}}
\put(3.00,50.00){\circle*{2.00}} \put(8.00,50.00){\circle*{2.00}}
\put(8.00,55.00){\circle*{2.00}} \put(3.00,55.00){\circle*{2.00}}
\put(48.00,50.00){\circle*{2.00}} \put(53.00,50.00){\circle*{2.00}}
\put(53.00,55.00){\circle*{2.00}} \put(48.00,55.00){\circle*{2.00}}
%\emline(55.00,45.00)(43.00,57.00)
\multiput(55.00,45.00)(-0.12,0.12){100}{\line(0,1){0.12}}
%\end
%\emline(43.00,57.00)(55.00,57.00)
\put(43.00,57.00){\line(1,0){12.00}}
%\end
%\emline(55.00,57.00)(55.00,45.00)
\put(55.00,57.00){\line(0,-1){12.00}}
%\end
%\emline(55.00,15.00)(43.00,3.00)
\multiput(55.00,15.00)(-0.12,-0.12){100}{\line(0,-1){0.12}}
%\end
%\emline(43.00,3.00)(55.00,3.00)
\put(43.00,3.00){\line(1,0){12.00}}
%\end
%\emline(55.00,3.00)(55.00,15.00)
\put(55.00,3.00){\line(0,1){12.00}}
%\end
%\emline(1.00,15.00)(13.00,3.00)
\multiput(1.00,15.00)(0.12,-0.12){100}{\line(0,-1){0.12}}
%\end
%\emline(13.00,3.00)(1.00,3.00)
\put(13.00,3.00){\line(-1,0){12.00}}
%\end
%\emline(1.00,3.00)(1.00,15.00)
\put(1.00,3.00){\line(0,1){12.00}}
%\end
%\emline(1.00,45.00)(13.00,57.00)
\multiput(1.00,45.00)(0.12,0.12){100}{\line(0,1){0.12}}
%\end
%\emline(13.00,57.00)(1.00,57.00)
\put(13.00,57.00){\line(-1,0){12.00}}
%\end
%\emline(1.00,57.00)(1.00,45.00)
\put(1.00,57.00){\line(0,-1){12.00}}
%\end
\put(28.00,2.00){\makebox(0,0)[cc]{{\large a}}}
\put(80.00,27.00){\circle*{2.00}} \put(80.00,33.00){\circle*{2.00}}
\put(86.00,30.00){\circle*{2.00}} \put(88.00,30.00){\circle*{2.00}}
%\emline(87.00,28.00)(87.00,32.00)
\put(87.00,28.00){\line(0,1){4.00}}
%\end
%\emline(87.00,32.00)(78.00,36.00)
\multiput(87.00,32.00)(-0.26,0.12){34}{\line(-1,0){0.26}}
%\end
%\emline(78.00,36.00)(78.00,24.00)
\put(78.00,36.00){\line(0,-1){12.00}}
%\end
%\emline(78.00,24.00)(87.00,28.00)
\multiput(78.00,24.00)(0.26,0.12){34}{\line(1,0){0.26}}
%\end
\put(109.00,29.00){\rule{2.00\unitlength}{2.00\unitlength}}
\put(133.00,47.00){\circle*{2.00}} \put(133.00,53.00){\circle*{2.00}}
\put(139.00,50.00){\circle*{2.00}} \put(141.00,50.00){\circle*{2.00}}
\put(133.00,7.00){\circle*{2.00}} \put(133.00,13.00){\circle*{2.00}}
\put(139.00,10.00){\circle*{2.00}} \put(141.00,10.00){\circle*{2.00}}
%\emline(143.00,8.00)(138.00,8.00)
\put(143.00,8.00){\line(-1,0){5.00}}
%\end
%\emline(138.00,8.00)(131.00,12.00)
\multiput(138.00,8.00)(-0.21,0.12){34}{\line(-1,0){0.21}}
%\end
%\emline(131.00,12.00)(131.00,16.00)
\put(131.00,12.00){\line(0,1){4.00}}
%\end
%\emline(131.00,16.00)(143.00,11.00)
\multiput(131.00,16.00)(0.29,-0.12){42}{\line(1,0){0.29}}
%\end
%\emline(143.00,11.00)(143.00,8.00)
\put(143.00,11.00){\line(0,-1){3.00}}
%\end
%\emline(143.00,52.00)(138.00,52.00)
\put(143.00,52.00){\line(-1,0){5.00}}
%\end
%\emline(138.00,52.00)(131.00,48.00)
\multiput(138.00,52.00)(-0.21,-0.12){34}{\line(-1,0){0.21}}
%\end
%\emline(131.00,48.00)(131.00,44.00)
\put(131.00,48.00){\line(0,-1){4.00}}
%\end
%\emline(131.00,44.00)(143.00,49.00)
\multiput(131.00,44.00)(0.29,0.12){42}{\line(1,0){0.29}}
%\end
%\emline(143.00,49.00)(143.00,52.00)
\put(143.00,49.00){\line(0,1){3.00}}
%\end
\put(128.00,12.00){\vector(-1,1){17.00}}
\put(111.00,31.00){\vector(1,1){17.00}}
%\emline(109.00,30.00)(92.00,30.00)
\put(109.00,30.00){\line(-1,0){17.00}}
%\end
\put(111.00,2.00){\makebox(0,0)[cc]{{\large b}}}
\end{picture}
\vspace{1.5cm} \unitlength 0.87mm \linethickness{0.4pt}
\begin{picture}(143.00,12.00)
\put(53.00,10.00){\circle*{2.00}} \put(61.00,10.00){\circle*{2.00}}
\put(63.00,10.00){\circle*{2.00}} \put(65.00,10.00){\circle*{2.00}}
\put(3.00,10.00){\circle*{2.00}} \put(11.00,10.00){\circle*{2.00}}
\put(13.00,10.00){\circle*{2.00}} \put(15.00,10.00){\circle*{2.00}}
%\emline(14.00,12.00)(14.00,8.00)
\put(14.00,12.00){\line(0,-1){4.00}}
%\end
%\emline(1.00,8.00)(1.00,12.00)
\put(1.00,8.00){\line(0,1){4.00}}
%\end
%\emline(59.00,8.00)(59.00,12.00)
\put(59.00,8.00){\line(0,1){4.00}}
%\end
%\emline(59.00,12.00)(67.00,12.00)
\put(59.00,12.00){\line(1,0){8.00}}
%\end
%\emline(67.00,12.00)(67.00,8.00)
\put(67.00,12.00){\line(0,-1){4.00}}
%\end
%\emline(67.00,8.00)(59.00,8.00)
\put(67.00,8.00){\line(-1,0){8.00}}
%\end
\put(33.00,9.00){\rule{2.00\unitlength}{2.00\unitlength}}
\put(34.00,2.00){\makebox(0,0)[cc]{{\large c}}}
%\emline(14.00,8.00)(1.00,8.00)
\put(14.00,8.00){\line(-1,0){13.00}}
%\end
%\emline(1.00,12.00)(14.00,12.00)
\put(1.00,12.00){\line(1,0){13.00}}
%\end
\put(83.00,10.00){\circle*{2.00}} \put(85.00,10.00){\circle*{2.00}}
\put(92.00,10.00){\circle*{2.00}} \put(94.00,10.00){\circle*{2.00}}
\put(130.00,10.00){\circle*{2.00}} \put(132.00,10.00){\circle*{2.00}}
\put(139.00,10.00){\circle*{2.00}} \put(141.00,10.00){\circle*{2.00}}
%\emline(131.00,8.00)(131.00,12.00)
\put(131.00,8.00){\line(0,1){4.00}}
%\end
%\emline(131.00,12.00)(143.00,12.00)
\put(131.00,12.00){\line(1,0){12.00}}
%\end
%\emline(143.00,12.00)(143.00,8.00)
\put(143.00,12.00){\line(0,-1){4.00}}
%\end
%\emline(143.00,8.00)(131.00,8.00)
\put(143.00,8.00){\line(-1,0){12.00}}
%\end
%\emline(93.00,8.00)(93.00,12.00)
\put(93.00,8.00){\line(0,1){4.00}}
%\end
%\emline(93.00,12.00)(81.00,12.00)
\put(93.00,12.00){\line(-1,0){12.00}}
%\end
%\emline(81.00,12.00)(81.00,8.00)
\put(81.00,12.00){\line(0,-1){4.00}}
%\end
%\emline(81.00,8.00)(93.00,8.00)
\put(81.00,8.00){\line(1,0){12.00}}
%\end
\put(111.00,9.00){\rule{2.00\unitlength}{2.00\unitlength}}
\put(113.00,10.00){\vector(1,0){13.00}} \put(98.00,10.00){\vector(1,0){13.00}}
\put(112.00,2.00){\makebox(0,0)[cc]{{\large d}}}
\put(19.00,10.00){\vector(1,0){14.00}} \put(35.00,10.00){\vector(1,0){14.00}}
\end{picture}
\end{center}
\caption{Order complexes for blocks in $\Diamond_3$} \label{str3dif}
\end{figure}

The block $[B]$ also is the space of a (trivial) fiber bundle, whose base is
the space of all configurations of 3 points in $S^1,$ one of which (the double
point) is distinguished, and the fiber is a star with 3 rays without endpoints.
These fibers and their endpoints are shown in Fig.~\ref{str3dif}b: two
right-hand endpoints correspond to the subspaces in $\K$ given by the
conditions of the form $\phi(x)=\phi(y), \phi'(x)=0$ and $\phi(x)=\phi(z),
\phi'(x)=0$, while the left endpoint corresponds to the equation
$\phi(x)=\phi(y)= \phi(z)$, cf. Example in \S~\ref{mainfil}.

The base of this bundle is diffeomorphic to the direct product $S^1 \times
B^2$, where $B^2$ is an open 2-dimensional disc. Monodromy along the basic
circle $S^1$ acts trivially on the bundle of 3-stars. Therefore we have the
following statement.

\prop \label{strB3} The Borel--Moore homology group of the reduced block $[B]$
is nontrivial only in dimensions $4$ and $3$ it is isomorphic to $\Z^2$ in both
these dimensions. \quad $\Box$ \eprop

In a similar way we get the following statements.

\prop \label{strC3} The Borel--Moore homology group of the reduced block $[C]$
is nontrivial only in dimensions $3$ and $2$ and is isomorphic to $\Z$ in both
these dimensions. \quad $\Box$ \eprop

\prop \label{strD3} The Borel--Moore homology group of the reduced block $[D]$
is nontrivial only in dimensions $3$ and $2$ and is isomorphic to $\Z$ in both
these dimensions. \quad $\Box$ \eprop

The corresponding order complexes are shown in Figs.~\ref{str3dif}c and
\ref{str3dif}d, respectively.
\medskip

\noindent {\bf Corollary.} {\em The term $E^1$ of the spectral sequence,
calculating the Borel--Moore homology group of the (not reduced) term
$\Diamond_3 \sm \Diamond_2$ of the main filtration of our resolved discriminant
and generated by the auxiliary filtration in this term, looks as is shown in
Fig.~\ref{ss3a} (i.e., all its cells $E^1_{p,q}$ other than six indicated there
are trivial).}

\begin{figure}
\begin{center}
\unitlength 1.00mm \linethickness{0.4pt}
\begin{picture}(52.00,24.00)
\put(12.00,6.00){\vector(0,1){18.00}} \put(12.00,6.00){\vector(1,0){40.00}}
\put(9.00,22.00){\makebox(0,0)[cc]{$q$}}
\put(5.00,9.00){\makebox(0,0)[cc]{$\Delta-6$}}
\put(5.00,15.00){\makebox(0,0)[cc]{$\Delta-5$}}
\put(17.00,2.00){\makebox(0,0)[cc]{2}} \put(17.00,9.00){\makebox(0,0)[cc]{${\bf
Z}^2$}} \put(17.00,15.00){\makebox(0,0)[cc]{${\bf Z}^2$}}
\put(26.00,9.00){\makebox(0,0)[cc]{${\bf Z}^2$}}
\put(26.00,15.00){\makebox(0,0)[cc]{${\bf Z}^2$}}
\put(35.00,15.00){\makebox(0,0)[cc]{${\bf Z}$}}
\put(35.00,9.00){\makebox(0,0)[cc]{${\bf Z}$}}
\put(35.00,2.00){\makebox(0,0)[cc]{4}} \put(26.00,2.00){\makebox(0,0)[cc]{3}}
\put(49.00,2.00){\makebox(0,0)[cc]{$a$}}
%\emline(12.00,12.00)(45.00,12.00)
\put(12.00,12.00){\line(1,0){33.00}}
%\end
%\emline(45.00,18.00)(12.00,18.00)
\put(45.00,18.00){\line(-1,0){33.00}}
%\end
%\emline(21.00,6.00)(21.00,21.00)
\put(21.00,6.00){\line(0,1){15.00}}
%\end
%\emline(30.00,21.00)(30.00,6.00)
\put(30.00,21.00){\line(0,-1){15.00}}
%\end
%\emline(39.00,6.00)(39.00,21.00)
\put(39.00,6.00){\line(0,1){15.00}}
%\end
\end{picture}
\end{center}
\caption{Auxiliary spectral sequence for the column $p=3$.} \label{ss3a}
\end{figure}
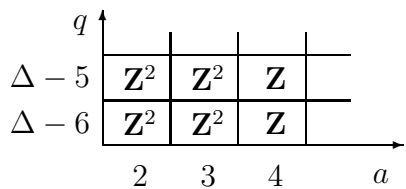

\prop \label{acyclABC} The Borel--Moore homology group of the subspace in
$\Diamond_3 \sm \Diamond_2$ formed only by the blocks A, B and C, is acyclic in
all dimensions. \eprop

This follows immediately from the shape of all these blocks and their
generators, and from accounting the limit positions of subspaces in $\K$
corresponding to the cliques of types A and B when they degenerate and form
configurations of types B and C, respectively.

E.g. let us consider the $(4)$-clique :: drawn at any endpoint of
Fig.~\ref{str3dif}a and let the two right-hand points of it move one towards
the other, forming at the last instant a clique of type B. The subspaces of
codimension 2 in $\K,$ corresponding to all endpoints of the cross, tend to
similar subspaces for endpoints of the star of Fig.~\ref{str3dif}b. Namely, to
both left-hand endpoints of the cross $\times$ there corresponds the unique
left-hand endpoint of the star, and other two endpoints "remain unmoved".
Therefore the boundary under this degeneration of the basic cycle shown in
Fig.~\ref{str3dif}a is equal to the basic cycle in $\bar H_{\Delta-3}(the \
block \ B)$ depicted by two arrows in Fig.~\ref{str3dif}b (i.e., swept out by
the fiber bundle of such 1-chains over entire base $S^1 \times B^2$ of this
block).

Other boundary operators $\bar H_*(A) \to \bar H_*(B)$ and $\bar H_*(B) \to
\bar H_*(C)$ can be considered in a similar way and give the assertion of
Proposition \ref{acyclABC}. \quad $\Box$
\medskip

\noindent {\bf Corollary.} {\em The Borel--Moore homology group of the space
$\Diamond_3 \sm \Diamond_2$ coincides with that of unique its block $D$
(described in Proposition \ref{strD3}).}
\medskip

This terminates the proof of Theorem \ref{s3}. \quad $\Box$

\subsection{Invariants of order 4.}

In this subsection, we shall be interested only in the $(\Delta-1)$-dimensional
homology group of the space $\Diamond_4 \sm \Diamond_3,$ which can provide
invariants of doodles.

In accordance with the Corollary of Proposition \ref{homeq2}, we shall consider
only blocks of complexity $4$ having at most one double point.

\prop \label{classcomp4} There are exactly 4 $\J$-blocks of complexity 4
corresponding to classes $\J$ of $A$-configurations (i.e. of $A$-cliques, all
whose points are geometrically distinct). Three of them correspond to
$(3,3)$-cliques consisting of 6 points in $S^1$ separated into two triples in
one of ways shown in Fig.~\ref{two}a. The fourth corresponds to the unique
class of $(5)$-configurations.

Also, there are exactly 5 blocks of complexity 4 corresponding to cliques with
exactly one point of multiplicity 2. Four of them also are of type $A=(3,3)$
and can be obtained from first two pictures in Fig.~\ref{two}a by some
degenerations, see Figs.~\ref{degen}a and \ref{degen}b respectively. The fifth
$\J$-block corresponds to $(5)$-cliques with exactly one double points, see
Fig.~\ref{degen}c. \quad $\Box$ \eprop

\begin{figure}
\unitlength 1.00mm \linethickness{0.4pt}
\begin{picture}(124.67,17.67)
%\circle(10.00,10.00){14.00}
\multiput(10.00,17.00)(0.79,-0.09){2}{\line(1,0){0.79}}
\multiput(11.59,16.82)(0.30,-0.11){5}{\line(1,0){0.30}}
\multiput(13.09,16.28)(0.17,-0.11){8}{\line(1,0){0.17}}
\multiput(14.43,15.42)(0.11,-0.11){10}{\line(0,-1){0.11}}
\multiput(15.54,14.27)(0.12,-0.20){7}{\line(0,-1){0.20}}
\multiput(16.37,12.91)(0.10,-0.30){5}{\line(0,-1){0.30}}
\multiput(16.86,11.39)(0.07,-0.80){2}{\line(0,-1){0.80}}
\multiput(17.00,9.80)(-0.11,-0.79){2}{\line(0,-1){0.79}}
\multiput(16.77,8.22)(-0.12,-0.30){5}{\line(0,-1){0.30}}
\multiput(16.19,6.73)(-0.11,-0.16){8}{\line(0,-1){0.16}}
\multiput(15.29,5.42)(-0.13,-0.12){9}{\line(-1,0){0.13}}
\multiput(14.11,4.34)(-0.20,-0.11){7}{\line(-1,0){0.20}}
\multiput(12.72,3.55)(-0.38,-0.11){4}{\line(-1,0){0.38}}
\put(11.19,3.10){\line(-1,0){1.59}}
\multiput(9.60,3.01)(-0.52,0.09){3}{\line(-1,0){0.52}}
\multiput(8.03,3.28)(-0.25,0.10){6}{\line(-1,0){0.25}}
\multiput(6.56,3.90)(-0.16,0.12){8}{\line(-1,0){0.16}}
\multiput(5.27,4.84)(-0.12,0.13){9}{\line(0,1){0.13}}
\multiput(4.22,6.05)(-0.11,0.20){7}{\line(0,1){0.20}}
\multiput(3.48,7.46)(-0.10,0.39){4}{\line(0,1){0.39}}
\put(3.07,9.00){\line(0,1){1.60}}
\multiput(3.03,10.60)(0.11,0.52){3}{\line(0,1){0.52}}
\multiput(3.34,12.16)(0.11,0.24){6}{\line(0,1){0.24}}
\multiput(4.01,13.61)(0.11,0.14){9}{\line(0,1){0.14}}
\multiput(4.98,14.88)(0.14,0.11){9}{\line(1,0){0.14}}
\multiput(6.22,15.89)(0.24,0.12){6}{\line(1,0){0.24}}
\multiput(7.65,16.59)(0.59,0.10){4}{\line(1,0){0.59}}
%\end
%\emline(4.00,7.00)(16.00,7.00)
\put(4.00,7.00){\line(1,0){12.00}}
%\end
%\emline(16.00,7.00)(10.00,3.00)
\multiput(16.00,7.00)(-0.18,-0.12){34}{\line(-1,0){0.18}}
%\end
%\emline(10.00,3.00)(4.00,7.00)
\multiput(10.00,3.00)(-0.18,0.12){34}{\line(-1,0){0.18}}
%\end
%\emline(4.00,13.00)(13.00,16.00)
\multiput(4.00,13.00)(0.35,0.12){26}{\line(1,0){0.35}}
%\end
\put(10.00,3.00){\circle*{1.33}} \put(16.00,7.00){\circle*{1.33}}
\put(4.00,7.00){\circle*{1.33}} \put(4.00,13.00){\circle*{1.33}}
%\circle(30.00,10.00){14.00}
\multiput(30.00,17.00)(0.79,-0.09){2}{\line(1,0){0.79}}
\multiput(31.59,16.82)(0.30,-0.11){5}{\line(1,0){0.30}}
\multiput(33.09,16.28)(0.17,-0.11){8}{\line(1,0){0.17}}
\multiput(34.43,15.42)(0.11,-0.11){10}{\line(0,-1){0.11}}
\multiput(35.54,14.27)(0.12,-0.20){7}{\line(0,-1){0.20}}
\multiput(36.37,12.91)(0.10,-0.30){5}{\line(0,-1){0.30}}
\multiput(36.86,11.39)(0.07,-0.80){2}{\line(0,-1){0.80}}
\multiput(37.00,9.80)(-0.11,-0.79){2}{\line(0,-1){0.79}}
\multiput(36.77,8.22)(-0.12,-0.30){5}{\line(0,-1){0.30}}
\multiput(36.19,6.73)(-0.11,-0.16){8}{\line(0,-1){0.16}}
\multiput(35.29,5.42)(-0.13,-0.12){9}{\line(-1,0){0.13}}
\multiput(34.11,4.34)(-0.20,-0.11){7}{\line(-1,0){0.20}}
\multiput(32.72,3.55)(-0.38,-0.11){4}{\line(-1,0){0.38}}
\put(31.19,3.10){\line(-1,0){1.59}}
\multiput(29.60,3.01)(-0.52,0.09){3}{\line(-1,0){0.52}}
\multiput(28.03,3.28)(-0.25,0.10){6}{\line(-1,0){0.25}}
\multiput(26.56,3.90)(-0.16,0.12){8}{\line(-1,0){0.16}}
\multiput(25.27,4.84)(-0.12,0.13){9}{\line(0,1){0.13}}
\multiput(24.22,6.05)(-0.11,0.20){7}{\line(0,1){0.20}}
\multiput(23.48,7.46)(-0.10,0.39){4}{\line(0,1){0.39}}
\put(23.07,9.00){\line(0,1){1.60}}
\multiput(23.03,10.60)(0.11,0.52){3}{\line(0,1){0.52}}
\multiput(23.34,12.16)(0.11,0.24){6}{\line(0,1){0.24}}
\multiput(24.01,13.61)(0.11,0.14){9}{\line(0,1){0.14}}
\multiput(24.98,14.88)(0.14,0.11){9}{\line(1,0){0.14}}
\multiput(26.22,15.89)(0.24,0.12){6}{\line(1,0){0.24}}
\multiput(27.65,16.59)(0.59,0.10){4}{\line(1,0){0.59}}
%\end
%\emline(24.00,7.00)(36.00,7.00)
\put(24.00,7.00){\line(1,0){12.00}}
%\end
%\emline(36.00,7.00)(30.00,3.00)
\multiput(36.00,7.00)(-0.18,-0.12){34}{\line(-1,0){0.18}}
%\end
%\emline(30.00,3.00)(24.00,7.00)
\multiput(30.00,3.00)(-0.18,0.12){34}{\line(-1,0){0.18}}
%\end
\put(30.00,3.00){\circle*{1.33}} \put(36.00,7.00){\circle*{1.33}}
\put(24.00,7.00){\circle*{1.33}}
%\circle(56.00,10.00){14.00}
\multiput(56.00,17.00)(0.79,-0.09){2}{\line(1,0){0.79}}
\multiput(57.59,16.82)(0.30,-0.11){5}{\line(1,0){0.30}}
\multiput(59.09,16.28)(0.17,-0.11){8}{\line(1,0){0.17}}
\multiput(60.43,15.42)(0.11,-0.11){10}{\line(0,-1){0.11}}
\multiput(61.54,14.27)(0.12,-0.20){7}{\line(0,-1){0.20}}
\multiput(62.37,12.91)(0.10,-0.30){5}{\line(0,-1){0.30}}
\multiput(62.86,11.39)(0.07,-0.80){2}{\line(0,-1){0.80}}
\multiput(63.00,9.80)(-0.11,-0.79){2}{\line(0,-1){0.79}}
\multiput(62.77,8.22)(-0.12,-0.30){5}{\line(0,-1){0.30}}
\multiput(62.19,6.73)(-0.11,-0.16){8}{\line(0,-1){0.16}}
\multiput(61.29,5.42)(-0.13,-0.12){9}{\line(-1,0){0.13}}
\multiput(60.11,4.34)(-0.20,-0.11){7}{\line(-1,0){0.20}}
\multiput(58.72,3.55)(-0.38,-0.11){4}{\line(-1,0){0.38}}
\put(57.19,3.10){\line(-1,0){1.59}}
\multiput(55.60,3.01)(-0.52,0.09){3}{\line(-1,0){0.52}}
\multiput(54.03,3.28)(-0.25,0.10){6}{\line(-1,0){0.25}}
\multiput(52.56,3.90)(-0.16,0.12){8}{\line(-1,0){0.16}}
\multiput(51.27,4.84)(-0.12,0.13){9}{\line(0,1){0.13}}
\multiput(50.22,6.05)(-0.11,0.20){7}{\line(0,1){0.20}}
\multiput(49.48,7.46)(-0.10,0.39){4}{\line(0,1){0.39}}
\put(49.07,9.00){\line(0,1){1.60}}
\multiput(49.03,10.60)(0.11,0.52){3}{\line(0,1){0.52}}
\multiput(49.34,12.16)(0.11,0.24){6}{\line(0,1){0.24}}
\multiput(50.01,13.61)(0.11,0.14){9}{\line(0,1){0.14}}
\multiput(50.98,14.88)(0.14,0.11){9}{\line(1,0){0.14}}
\multiput(52.22,15.89)(0.24,0.12){6}{\line(1,0){0.24}}
\multiput(53.65,16.59)(0.59,0.10){4}{\line(1,0){0.59}}
%\end
%\emline(56.00,3.00)(50.00,7.00)
\multiput(56.00,3.00)(-0.18,0.12){34}{\line(-1,0){0.18}}
%\end
\put(56.00,3.00){\circle*{1.33}} \put(62.00,7.00){\circle*{1.33}}
\put(50.00,7.00){\circle*{1.33}}
%\circle(76.00,10.00){14.00}
\multiput(76.00,17.00)(0.79,-0.09){2}{\line(1,0){0.79}}
\multiput(77.59,16.82)(0.30,-0.11){5}{\line(1,0){0.30}}
\multiput(79.09,16.28)(0.17,-0.11){8}{\line(1,0){0.17}}
\multiput(80.43,15.42)(0.11,-0.11){10}{\line(0,-1){0.11}}
\multiput(81.54,14.27)(0.12,-0.20){7}{\line(0,-1){0.20}}
\multiput(82.37,12.91)(0.10,-0.30){5}{\line(0,-1){0.30}}
\multiput(82.86,11.39)(0.07,-0.80){2}{\line(0,-1){0.80}}
\multiput(83.00,9.80)(-0.11,-0.79){2}{\line(0,-1){0.79}}
\multiput(82.77,8.22)(-0.12,-0.30){5}{\line(0,-1){0.30}}
\multiput(82.19,6.73)(-0.11,-0.16){8}{\line(0,-1){0.16}}
\multiput(81.29,5.42)(-0.13,-0.12){9}{\line(-1,0){0.13}}
\multiput(80.11,4.34)(-0.20,-0.11){7}{\line(-1,0){0.20}}
\multiput(78.72,3.55)(-0.38,-0.11){4}{\line(-1,0){0.38}}
\put(77.19,3.10){\line(-1,0){1.59}}
\multiput(75.60,3.01)(-0.52,0.09){3}{\line(-1,0){0.52}}
\multiput(74.03,3.28)(-0.25,0.10){6}{\line(-1,0){0.25}}
\multiput(72.56,3.90)(-0.16,0.12){8}{\line(-1,0){0.16}}
\multiput(71.27,4.84)(-0.12,0.13){9}{\line(0,1){0.13}}
\multiput(70.22,6.05)(-0.11,0.20){7}{\line(0,1){0.20}}
\multiput(69.48,7.46)(-0.10,0.39){4}{\line(0,1){0.39}}
\put(69.07,9.00){\line(0,1){1.60}}
\multiput(69.03,10.60)(0.11,0.52){3}{\line(0,1){0.52}}
\multiput(69.34,12.16)(0.11,0.24){6}{\line(0,1){0.24}}
\multiput(70.01,13.61)(0.11,0.14){9}{\line(0,1){0.14}}
\multiput(70.98,14.88)(0.14,0.11){9}{\line(1,0){0.14}}
\multiput(72.22,15.89)(0.24,0.12){6}{\line(1,0){0.24}}
\multiput(73.65,16.59)(0.59,0.10){4}{\line(1,0){0.59}}
%\end
\put(82.00,7.00){\circle*{1.33}} \put(12.33,16.34){\circle*{1.33}}
\put(13.67,16.00){\circle*{1.33}} \put(36.00,13.00){\circle*{1.33}}
%\emline(36.00,13.00)(26.33,15.67)
\multiput(36.00,13.00)(-0.42,0.12){23}{\line(-1,0){0.42}}
%\end
\put(27.00,16.00){\circle*{1.33}} \put(25.99,15.66){\circle*{1.33}}
\put(62.00,13.33){\circle*{1.33}}
%\emline(62.00,13.33)(56.00,3.00)
\multiput(62.00,13.33)(-0.12,-0.21){50}{\line(0,-1){0.21}}
%\end
%\emline(50.00,7.00)(62.00,13.33)
\multiput(50.00,7.00)(0.23,0.12){53}{\line(1,0){0.23}}
%\end
%\emline(62.00,7.00)(51.33,15.00)
\multiput(62.00,7.00)(-0.16,0.12){67}{\line(-1,0){0.16}}
%\end
\put(52.00,15.67){\circle*{1.33}} \put(51.00,14.66){\circle*{1.33}}
\put(82.00,13.00){\circle*{1.33}}
%\emline(82.00,13.00)(72.67,3.67)
\multiput(82.00,13.00)(-0.12,-0.12){78}{\line(-1,0){0.12}}
%\end
\put(72.00,4.00){\circle*{1.33}} \put(73.34,3.66){\circle*{1.33}}
%\emline(82.00,7.00)(76.00,17.00)
\multiput(82.00,7.00)(-0.12,0.20){50}{\line(0,1){0.20}}
%\end
%\emline(76.00,17.00)(69.67,13.00)
\multiput(76.00,17.00)(-0.19,-0.12){34}{\line(-1,0){0.19}}
%\end
%\emline(69.67,13.00)(82.00,7.00)
\multiput(69.67,13.00)(0.25,-0.12){50}{\line(1,0){0.25}}
%\end
\put(76.00,16.67){\circle*{1.33}} \put(70.00,13.00){\circle*{1.33}}
%\circle(117.00,10.00){14.00}
\multiput(117.00,17.00)(0.79,-0.09){2}{\line(1,0){0.79}}
\multiput(118.59,16.82)(0.30,-0.11){5}{\line(1,0){0.30}}
\multiput(120.09,16.28)(0.17,-0.11){8}{\line(1,0){0.17}}
\multiput(121.43,15.42)(0.11,-0.11){10}{\line(0,-1){0.11}}
\multiput(122.54,14.27)(0.12,-0.20){7}{\line(0,-1){0.20}}
\multiput(123.37,12.91)(0.10,-0.30){5}{\line(0,-1){0.30}}
\multiput(123.86,11.39)(0.07,-0.80){2}{\line(0,-1){0.80}}
\multiput(124.00,9.80)(-0.11,-0.79){2}{\line(0,-1){0.79}}
\multiput(123.77,8.22)(-0.12,-0.30){5}{\line(0,-1){0.30}}
\multiput(123.19,6.73)(-0.11,-0.16){8}{\line(0,-1){0.16}}
\multiput(122.29,5.42)(-0.13,-0.12){9}{\line(-1,0){0.13}}
\multiput(121.11,4.34)(-0.20,-0.11){7}{\line(-1,0){0.20}}
\multiput(119.72,3.55)(-0.38,-0.11){4}{\line(-1,0){0.38}}
\put(118.19,3.10){\line(-1,0){1.59}}
\multiput(116.60,3.01)(-0.52,0.09){3}{\line(-1,0){0.52}}
\multiput(115.03,3.28)(-0.25,0.10){6}{\line(-1,0){0.25}}
\multiput(113.56,3.90)(-0.16,0.12){8}{\line(-1,0){0.16}}
\multiput(112.27,4.84)(-0.12,0.13){9}{\line(0,1){0.13}}
\multiput(111.22,6.05)(-0.11,0.20){7}{\line(0,1){0.20}}
\multiput(110.48,7.46)(-0.10,0.39){4}{\line(0,1){0.39}}
\put(110.07,9.00){\line(0,1){1.60}}
\multiput(110.03,10.60)(0.11,0.52){3}{\line(0,1){0.52}}
\multiput(110.34,12.16)(0.11,0.24){6}{\line(0,1){0.24}}
\multiput(111.01,13.61)(0.11,0.14){9}{\line(0,1){0.14}}
\multiput(111.98,14.88)(0.14,0.11){9}{\line(1,0){0.14}}
\multiput(113.22,15.89)(0.24,0.12){6}{\line(1,0){0.24}}
\multiput(114.65,16.59)(0.59,0.10){4}{\line(1,0){0.59}}
%\end
\put(124.00,10.00){\circle*{1.33}} \put(110.00,10.00){\circle*{1.33}}
\put(117.00,17.00){\circle*{1.33}} \put(116.33,3.00){\circle*{1.33}}
\put(117.67,3.00){\circle*{1.33}} \put(109.00,1.00){\makebox(0,0)[cc]{{\large
c}}} \put(66.00,1.00){\makebox(0,0)[cc]{{\large b}}}
\put(20.00,1.00){\makebox(0,0)[cc]{{\large a}}}
\end{picture}
\caption{Degenerated blocks of complexity 4} \label{degen}
\end{figure}
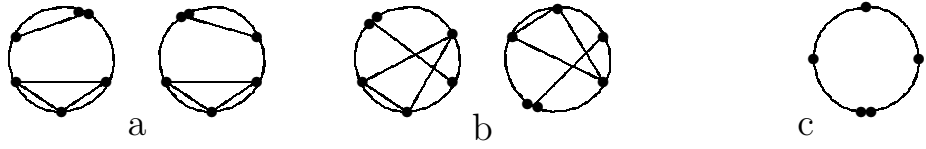

Let us study all these $\J$-blocks and their homology groups.

\prop \label{two_hom} All three $\J$-blocks corresponding to three pictures of
Fig.~\ref{two}a are smooth orientable manifolds diffeomorphic to $S^1 \times
\R^{\Delta-2},$ in particular any of them has only two nontrivial Borel--Moore
homology groups $\bar H_{\Delta-1} \simeq \Z \simeq \bar H_{\Delta-2}.$ \quad
$\Box$ \eprop

More precisely, in the first and the third (respectively, in the second) case
the corresponding space $\J$ of equivalent cliques is diffeomorphic to a
nonorientable (respectively, orientable) fiber bundle over $S^1$ with fiber
$B^5.$ The order complexes $\Diamond(J)$ in all cases are homeomorphic to open
intervals, whose endpoints correspond to subspaces of codimension 2 in $\K$
defined by the $(3)$-cliques forming the triangles, and the center corresponds
to the subspace of codimension 4, defined by their intersection. The bundle of
these intervals is (non)orientable exactly in the same cases when the
corresponding configuration space is.

\prop \label{three_hom} The $(\Delta-1)$-dimensional Borel--Moore homology
group of the $\J$-block in $\Diamond_4 \sm \Diamond_3,$ corresponding to the
unique class $\J$ of $(5)$-{\em configurations}, is equal to $\Z^2.$ \eprop

(I thank very much A.~B.~Merkov, who proved this proposition, and also
suggested the following notation, convenient for the homological study of such
blocks.)

\begin{proof} Let $J$ be a configuration of 5 different points in
$S^1$. The corresponding order complex $\Diamond(J)$ is two-dimensional. Its 20
simplices of dimension 2 are in a natural one-to-one correspondence with the
triples of the form \{some 3 points of $J$; some 4 points containing these
three; all five points\}. Denote such a simplex by the arrow, connecting two
points {\em not participating} in the first triple and directed towards the
point not participating in the quadruple. The one-dimensional simplices of the
same order complex $\Diamond(J)$ are the segments of the following three kinds.

A) connecting (the vertex corresponding to the subspace in $\K$ defined by) a
triple of our points and (the vertex corresponding to its subspace defined by)
a quadruple containing this triple. Such edges do not belong to $\Diamond_4$
and are not interesting for us.

B) connecting (the vertex corresponding to) a triple and (that corresponding
to) the maximal element $\chi(J) \in \Diamond(J)$ (defined by the entire
$(5)$-clique $J$). These 10 edges are denoted by non-oriented edges connecting
two points not participating in the triple.

C) connecting (the vertices corresponding to) a quadruple and $\chi(J)$. These
5 edges are denoted by marking the point not participating in the quadruple.
\medskip

In this notation, the boundary of an arrow (i.e., a 2-simplex of $\Diamond(J)$)
is equal to the edge, obtained from this arrow by forgetting the orientation,
minus the endpoint of the arrow.

\begin{lemma}
The Borel--Moore homology group of this complex is located in dimension $2$ and
is isomorphic to $\Z^6.$
\end{lemma}

\begin{proof}
Indeed, the generating it 2-cycles look as follows.

First of all, any arrow can appear in such a cycle only together with its
opposite, taken with opposite coefficient: otherwise the boundary of this cycle
will contain an edge of type B) with a non-zero coefficient. Thus it is
sufficient to consider the 10-dimensional group, generated by the linear
combinations of the form \{an arrow minus its opposite\}. Such an element will
be depicted by a double arrow directed as the first arrow in this combination,
see Fig.~\ref{two}b. A linear combination of such double arrows is a cycle of
the complex of closed chains of $\Diamond(J) \sm \Diamond_3$ if and only if the
correspondingly oriented segments form a cycle of the complete graph with 5
vertices. The group $H_1$ of this complete graph is isomorphic to $\Z^6,$ and
lemma is proved. \end{proof}

Further, our $\J$-block is the space of a fiber bundle, whose base is the
configuration space $B(S^1,5) \cong S^1 \times B^4,$ and the fiber over the
configuration $J$ is the direct product of the oriented
$(\Delta-8)$-dimensional subspace $\chi(J) \subset \K$ and the complex
$\Diamond(J) \sm \Diamond_3$. The monodromy over the generator $S^1$ of the
fundamental group of the base acts on this complex (and its homology) as a
cyclic permutation of 5 vertices. Thus by the Wang exact sequence the group
considered in Proposition \ref{three_hom} is generated exactly by all cycles of
a complete 5-graph which are invariant under this action. This group is
two-dimensional; its generators are shown in Fig.~\ref{two}b.
\end{proof}

We have found all possible generators of the group
\begin{equation}
\label{sig4} \bar H_{\Delta-1}(\Diamond_4 \sm \Diamond_3),
\end{equation}
namely the following statement holds.

\prop \label{gensig4} Any element of the group (\ref{sig4}) is a linear
combination of five chains shown in Fig.~\ref{two}. \quad $\Box$ \eprop

Now let us study the boundaries of these chains in other blocks.

\prop \label{trivtwo} Two chains corresponding to two left pictures in
Fig.~\ref{two}a cannot participate in an element of the group (\ref{sig4}) with
nonzero coefficients. \eprop

Indeed, the boundary of the first (respectively, the second) of them contains
the sum of generators of $(\Delta-2)$-dimensional homology groups of two blocks
shown in Fig.~\ref{degen}a (respectively, \ref{degen}b). These generators do
not appear in the boundaries of any other of our 5 chains. \quad $\Box$
\medskip

\prop \label{triv3} Two chains corresponding to two pictures in Fig.~\ref{two}b
cannot participate in an element of the group (\ref{sig4}) with nonzero
coefficients. \eprop

To prove this proposition, let us consider the homology group of the
$\J$-block, corresponding to $(5)$-cliques $J = (x,x,y,z,w)$, as shown in
Fig.~\ref{degen}c. The corresponding order complex $\Diamond(J)$ again is
two-dimensional.

\prop \label{rel4} For any $(5)$-clique $J$ consisting of exactly 4
geometrically different points in $S^1,$ the group $\bar H_*(\Diamond(J) \sm
\Diamond_3)$ is concentrated in dimension 2 and is isomorphic to $\Z^3$. The
Borel--Moore homology group of the corresponding block in $\Diamond_4 \sm
\Diamond_3$ is concentrated in dimension $\Delta-2$ and also is isomorphic to
$\Z^3$. \eprop

\begin{proof}
Exactly as in the proof of Proposition \ref{three_hom}, almost all
two-di\-men\-si\-onal simplices of the complex $\Diamond(J)$ are naturally
depicted by arrows connecting some of 4 geometrically distinct points of the
clique $J$. (The unique extra triangle is the triple of subspaces, whose first
element is defined by three points of multiplicity 1: it should be denoted by a
loop edge, connecting the double point with itself. This triangle cannot
participate with non-zero coefficient in any cycle of the complex $\Diamond(J)
\sm \Diamond_3$.)

Again, the cycles of this complex are depicted by double arrows, forming cycles
(in the usual sense) of the complete graph on our 4 vertices. This proves the
first assertion of Proposition \ref{rel4}.

The entire $J$-block is the space of a fiber bundle over the space of all
cliques of this type (which is diffeomorphic to $S^1 \times B^3$), namely, of a
fibered product of an orientable $(\Delta-8)$-dimensional vector bundle and the
(trivial) bundle of complexes $\Diamond(J)\sm \Diamond_3$. This proves the last
assertion of Proposition. Moreover, the $(\Delta-2)$-dimensional cycles of the
block are in a one-to-one (K\"unneth) correspondence with 2-dimensional cycles
of $\Diamond(J) \sm \Diamond_3,$ where $J$ is any clique of this class.
\end{proof}

\prop \label{bound3} The boundaries in this block of two basic
$(\Delta-1)$-dimensional cycles, shown in Fig.~\ref{two}b, are the two cycles
shown in Fig.~\ref{boun3}a. \quad $\Box$ \eprop

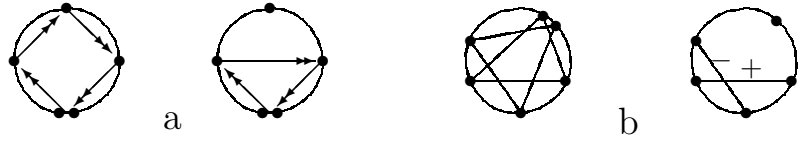
\begin{figure}
\begin{center}
\unitlength 1.00mm \linethickness{0.4pt}
\begin{picture}(107.00,17.67)
%\circle(10.00,10.00){14.00}
\multiput(10.00,17.00)(0.79,-0.09){2}{\line(1,0){0.79}}
\multiput(11.59,16.82)(0.30,-0.11){5}{\line(1,0){0.30}}
\multiput(13.09,16.28)(0.17,-0.11){8}{\line(1,0){0.17}}
\multiput(14.43,15.42)(0.11,-0.11){10}{\line(0,-1){0.11}}
\multiput(15.54,14.27)(0.12,-0.20){7}{\line(0,-1){0.20}}
\multiput(16.37,12.91)(0.10,-0.30){5}{\line(0,-1){0.30}}
\multiput(16.86,11.39)(0.07,-0.80){2}{\line(0,-1){0.80}}
\multiput(17.00,9.80)(-0.11,-0.79){2}{\line(0,-1){0.79}}
\multiput(16.77,8.22)(-0.12,-0.30){5}{\line(0,-1){0.30}}
\multiput(16.19,6.73)(-0.11,-0.16){8}{\line(0,-1){0.16}}
\multiput(15.29,5.42)(-0.13,-0.12){9}{\line(-1,0){0.13}}
\multiput(14.11,4.34)(-0.20,-0.11){7}{\line(-1,0){0.20}}
\multiput(12.72,3.55)(-0.38,-0.11){4}{\line(-1,0){0.38}}
\put(11.19,3.10){\line(-1,0){1.59}}
\multiput(9.60,3.01)(-0.52,0.09){3}{\line(-1,0){0.52}}
\multiput(8.03,3.28)(-0.25,0.10){6}{\line(-1,0){0.25}}
\multiput(6.56,3.90)(-0.16,0.12){8}{\line(-1,0){0.16}}
\multiput(5.27,4.84)(-0.12,0.13){9}{\line(0,1){0.13}}
\multiput(4.22,6.05)(-0.11,0.20){7}{\line(0,1){0.20}}
\multiput(3.48,7.46)(-0.10,0.39){4}{\line(0,1){0.39}}
\put(3.07,9.00){\line(0,1){1.60}}
\multiput(3.03,10.60)(0.11,0.52){3}{\line(0,1){0.52}}
\multiput(3.34,12.16)(0.11,0.24){6}{\line(0,1){0.24}}
\multiput(4.01,13.61)(0.11,0.14){9}{\line(0,1){0.14}}
\multiput(4.98,14.88)(0.14,0.11){9}{\line(1,0){0.14}}
\multiput(6.22,15.89)(0.24,0.12){6}{\line(1,0){0.24}}
\multiput(7.65,16.59)(0.59,0.10){4}{\line(1,0){0.59}}
%\end
\put(9.00,3.00){\circle*{1.33}} \put(11.00,3.00){\circle*{1.33}}
\put(17.00,10.00){\circle*{1.33}} \put(10.00,17.00){\circle*{1.33}}
\put(3.00,10.00){\circle*{1.33}}
%\circle(37.00,10.00){14.00}
\multiput(37.00,17.00)(0.79,-0.09){2}{\line(1,0){0.79}}
\multiput(38.59,16.82)(0.30,-0.11){5}{\line(1,0){0.30}}
\multiput(40.09,16.28)(0.17,-0.11){8}{\line(1,0){0.17}}
\multiput(41.43,15.42)(0.11,-0.11){10}{\line(0,-1){0.11}}
\multiput(42.54,14.27)(0.12,-0.20){7}{\line(0,-1){0.20}}
\multiput(43.37,12.91)(0.10,-0.30){5}{\line(0,-1){0.30}}
\multiput(43.86,11.39)(0.07,-0.80){2}{\line(0,-1){0.80}}
\multiput(44.00,9.80)(-0.11,-0.79){2}{\line(0,-1){0.79}}
\multiput(43.77,8.22)(-0.12,-0.30){5}{\line(0,-1){0.30}}
\multiput(43.19,6.73)(-0.11,-0.16){8}{\line(0,-1){0.16}}
\multiput(42.29,5.42)(-0.13,-0.12){9}{\line(-1,0){0.13}}
\multiput(41.11,4.34)(-0.20,-0.11){7}{\line(-1,0){0.20}}
\multiput(39.72,3.55)(-0.38,-0.11){4}{\line(-1,0){0.38}}
\put(38.19,3.10){\line(-1,0){1.59}}
\multiput(36.60,3.01)(-0.52,0.09){3}{\line(-1,0){0.52}}
\multiput(35.03,3.28)(-0.25,0.10){6}{\line(-1,0){0.25}}
\multiput(33.56,3.90)(-0.16,0.12){8}{\line(-1,0){0.16}}
\multiput(32.27,4.84)(-0.12,0.13){9}{\line(0,1){0.13}}
\multiput(31.22,6.05)(-0.11,0.20){7}{\line(0,1){0.20}}
\multiput(30.48,7.46)(-0.10,0.39){4}{\line(0,1){0.39}}
\put(30.07,9.00){\line(0,1){1.60}}
\multiput(30.03,10.60)(0.11,0.52){3}{\line(0,1){0.52}}
\multiput(30.34,12.16)(0.11,0.24){6}{\line(0,1){0.24}}
\multiput(31.01,13.61)(0.11,0.14){9}{\line(0,1){0.14}}
\multiput(31.98,14.88)(0.14,0.11){9}{\line(1,0){0.14}}
\multiput(33.22,15.89)(0.24,0.12){6}{\line(1,0){0.24}}
\multiput(34.65,16.59)(0.59,0.10){4}{\line(1,0){0.59}}
%\end
\put(36.00,3.00){\circle*{1.33}} \put(38.00,3.00){\circle*{1.33}}
\put(44.00,10.00){\circle*{1.33}} \put(37.00,17.00){\circle*{1.33}}
\put(30.00,10.00){\circle*{1.33}} \put(10.00,3.00){\vector(-1,1){6.00}}
\put(10.00,3.00){\vector(-1,1){5.00}} \put(3.00,10.00){\vector(1,1){6.00}}
\put(3.00,10.00){\vector(1,1){5.00}} \put(10.00,17.00){\vector(1,-1){6.00}}
\put(10.00,17.00){\vector(1,-1){5.00}} \put(17.00,10.00){\vector(-1,-1){6.00}}
\put(17.00,10.00){\vector(-1,-1){5.00}} \put(37.00,3.00){\vector(-1,1){6.00}}
\put(37.00,3.00){\vector(-1,1){5.00}} \put(30.00,10.00){\vector(1,0){13.00}}
\put(30.00,10.00){\vector(1,0){12.00}} \put(44.00,10.00){\vector(-1,-1){6.00}}
\put(44.00,10.00){\vector(-1,-1){5.00}}
\put(24.00,2.00){\makebox(0,0)[cc]{{\large a}}}
%\circle(70.00,10.00){14.00}
\multiput(70.00,17.00)(0.79,-0.09){2}{\line(1,0){0.79}}
\multiput(71.59,16.82)(0.30,-0.11){5}{\line(1,0){0.30}}
\multiput(73.09,16.28)(0.17,-0.11){8}{\line(1,0){0.17}}
\multiput(74.43,15.42)(0.11,-0.11){10}{\line(0,-1){0.11}}
\multiput(75.54,14.27)(0.12,-0.20){7}{\line(0,-1){0.20}}
\multiput(76.37,12.91)(0.10,-0.30){5}{\line(0,-1){0.30}}
\multiput(76.86,11.39)(0.07,-0.80){2}{\line(0,-1){0.80}}
\multiput(77.00,9.80)(-0.11,-0.79){2}{\line(0,-1){0.79}}
\multiput(76.77,8.22)(-0.12,-0.30){5}{\line(0,-1){0.30}}
\multiput(76.19,6.73)(-0.11,-0.16){8}{\line(0,-1){0.16}}
\multiput(75.29,5.42)(-0.13,-0.12){9}{\line(-1,0){0.13}}
\multiput(74.11,4.34)(-0.20,-0.11){7}{\line(-1,0){0.20}}
\multiput(72.72,3.55)(-0.38,-0.11){4}{\line(-1,0){0.38}}
\put(71.19,3.10){\line(-1,0){1.59}}
\multiput(69.60,3.01)(-0.52,0.09){3}{\line(-1,0){0.52}}
\multiput(68.03,3.28)(-0.25,0.10){6}{\line(-1,0){0.25}}
\multiput(66.56,3.90)(-0.16,0.12){8}{\line(-1,0){0.16}}
\multiput(65.27,4.84)(-0.12,0.13){9}{\line(0,1){0.13}}
\multiput(64.22,6.05)(-0.11,0.20){7}{\line(0,1){0.20}}
\multiput(63.48,7.46)(-0.10,0.39){4}{\line(0,1){0.39}}
\put(63.07,9.00){\line(0,1){1.60}}
\multiput(63.03,10.60)(0.11,0.52){3}{\line(0,1){0.52}}
\multiput(63.34,12.16)(0.11,0.24){6}{\line(0,1){0.24}}
\multiput(64.01,13.61)(0.11,0.14){9}{\line(0,1){0.14}}
\multiput(64.98,14.88)(0.14,0.11){9}{\line(1,0){0.14}}
\multiput(66.22,15.89)(0.24,0.12){6}{\line(1,0){0.24}}
\multiput(67.65,16.59)(0.59,0.10){4}{\line(1,0){0.59}}
%\end
\put(63.67,7.33){\circle*{1.33}} \put(76.33,7.33){\circle*{1.33}}
\put(63.67,12.67){\circle*{1.33}} \put(70.33,3.00){\circle*{1.33}}
\put(73.33,16.00){\circle*{1.33}} \put(75.00,14.67){\circle*{1.33}}
%\emline(75.00,14.67)(63.67,12.67)
\multiput(75.00,14.67)(-0.67,-0.12){17}{\line(-1,0){0.67}}
%\end
%\emline(63.67,12.67)(70.33,3.00)
\multiput(63.67,12.67)(0.12,-0.17){56}{\line(0,-1){0.17}}
%\end
%\emline(70.33,3.00)(75.00,14.67)
\multiput(70.33,3.00)(0.12,0.30){39}{\line(0,1){0.30}}
%\end
%\emline(76.33,7.33)(63.67,7.33)
\put(76.33,7.33){\line(-1,0){12.66}}
%\end
%\emline(63.67,7.33)(73.33,16.00)
\multiput(63.67,7.33)(0.13,0.12){73}{\line(1,0){0.13}}
%\end
%\emline(73.33,16.00)(76.33,7.33)
\multiput(73.33,16.00)(0.12,-0.33){26}{\line(0,-1){0.33}}
%\end
%\circle(100.00,10.00){14.00}
\multiput(100.00,17.00)(0.79,-0.09){2}{\line(1,0){0.79}}
\multiput(101.59,16.82)(0.30,-0.11){5}{\line(1,0){0.30}}
\multiput(103.09,16.28)(0.17,-0.11){8}{\line(1,0){0.17}}
\multiput(104.43,15.42)(0.11,-0.11){10}{\line(0,-1){0.11}}
\multiput(105.54,14.27)(0.12,-0.20){7}{\line(0,-1){0.20}}
\multiput(106.37,12.91)(0.10,-0.30){5}{\line(0,-1){0.30}}
\multiput(106.86,11.39)(0.07,-0.80){2}{\line(0,-1){0.80}}
\multiput(107.00,9.80)(-0.11,-0.79){2}{\line(0,-1){0.79}}
\multiput(106.77,8.22)(-0.12,-0.30){5}{\line(0,-1){0.30}}
\multiput(106.19,6.73)(-0.11,-0.16){8}{\line(0,-1){0.16}}
\multiput(105.29,5.42)(-0.13,-0.12){9}{\line(-1,0){0.13}}
\multiput(104.11,4.34)(-0.20,-0.11){7}{\line(-1,0){0.20}}
\multiput(102.72,3.55)(-0.38,-0.11){4}{\line(-1,0){0.38}}
\put(101.19,3.10){\line(-1,0){1.59}}
\multiput(99.60,3.01)(-0.52,0.09){3}{\line(-1,0){0.52}}
\multiput(98.03,3.28)(-0.25,0.10){6}{\line(-1,0){0.25}}
\multiput(96.56,3.90)(-0.16,0.12){8}{\line(-1,0){0.16}}
\multiput(95.27,4.84)(-0.12,0.13){9}{\line(0,1){0.13}}
\multiput(94.22,6.05)(-0.11,0.20){7}{\line(0,1){0.20}}
\multiput(93.48,7.46)(-0.10,0.39){4}{\line(0,1){0.39}}
\put(93.07,9.00){\line(0,1){1.60}}
\multiput(93.03,10.60)(0.11,0.52){3}{\line(0,1){0.52}}
\multiput(93.34,12.16)(0.11,0.24){6}{\line(0,1){0.24}}
\multiput(94.01,13.61)(0.11,0.14){9}{\line(0,1){0.14}}
\multiput(94.98,14.88)(0.14,0.11){9}{\line(1,0){0.14}}
\multiput(96.22,15.89)(0.24,0.12){6}{\line(1,0){0.24}}
\multiput(97.65,16.59)(0.59,0.10){4}{\line(1,0){0.59}}
%\end
\put(93.67,7.33){\circle*{1.33}} \put(106.33,7.33){\circle*{1.33}}
\put(93.67,12.67){\circle*{1.33}} \put(100.33,3.00){\circle*{1.33}}
\put(104.33,15.33){\circle*{1.33}}
%\emline(106.33,7.33)(93.67,7.33)
\put(106.33,7.33){\line(-1,0){12.66}}
%\end
%\emline(93.67,12.67)(100.33,3.00)
\multiput(93.67,12.67)(0.12,-0.17){56}{\line(0,-1){0.17}}
%\end
\put(84.67,2.00){\makebox(0,0)[cc]{{\large b}}}
\put(101.00,9.00){\makebox(0,0)[cc]{\small +}}
\put(97.00,10.00){\makebox(0,0)[cc]{{\small $-$}}}
\end{picture}
\end{center}
\caption{Boundaries of basic cycles of Fig.~\protect\ref{two}b and
Fig.~\protect\ref{md}} \label{boun3}
\end{figure}

\prop \label{mdcycle} The chain in $\Diamond_4,$ shown in Fig.~\ref{md}
($\equiv$ the right-hand picture in Fig.~\ref{two}a), defines a cycle in
$\Diamond_4 \sm \Diamond_3.$ \eprop

\begin{proof}
The $(3,3)$-cliques of this type {\Large $*$} can degenerate only in the
following way: some two neighboring points of {\em different} groups coincide,
thus forming a $(5)$-con\-fi\-gu\-ra\-tion, see the left picture of
Fig.~\ref{boun3}b. Given a $(5)$-con\-fi\-gu\-ra\-tion $J$, the $\J'$-block of
the type {\Large $*$} adjoins the corresponding complex $\Diamond(J)$ exactly 5
times, because such coincidence can happen at any of its 5 points. The boundary
position of complexes $\Diamond(J'),$ when $J' \in ${\Large $*$} tends to $J$
in the way shown in Fig.~\ref{boun3}b left, is (in the notation used in the
proof of Proposition \ref{three_hom}) equal to the difference of two
1-dimensional simplices in $\Diamond(J)$ denoted by two edges in
Fig.~\ref{boun3}b right. The sum of such differences over all 5 vertices of the
configuration $J$ is equal to 0, and Proposition \ref{mdcycle} is proved.
\end{proof}

\begin{remark}
A much more general fact was proved in \cite{Merx-2}: any "horizontal" boundary
operator $d^1$ of the auxiliary spectral sequence, corresponding to the
collision of two points of two different groups, always is trivial.
\end{remark}

Summarizing the Propositions \ref{classcomp4}--\ref{mdcycle}, we get the
following statement.

\thm The group $\bar H_{\Delta-1}(\Diamond_4 \sm \Diamond_3)$ is isomorphic to
$\Z$ and is generated by the fundamental cycle of the $\J$-block corresponding
to the $(3,3)$-cliques shown in Fig.~\ref{md}. \etheorem

As the $(\Delta-1)$- and $(\Delta-2)$-dimensional Borel--Moore homology groups
of both spaces $\Diamond_2 $ and $\Diamond_3 \sm \Diamond_2$ are trivial (see
Proposition \ref{si2} and Theorem \ref{s3}), this implies Theorem \ref{doodinv}
of the Introduction.

\subsection{A nontrivial doodle.}
\label{merxdood}

The theory of finite-order invariants provides a method of constructing a
priori nontrivial (and nonequivalent) objects. E.g., imagine that we do not
know any nontrivial knot in $\R^3$ and wish to construct it. To do it, we can
calculate the simplest finite-order knot invariant (given by the chord diagram
$\bigoplus$), then draw the simplest singular knot respecting this diagram
(i.e. having two transverse selfintersections), and then consider four knots
obtained from it by all possible local resolutions of both these points. At
least one of obtained knots surely will be nonequivalent to the others (and
indeed, if we do all this in the simplest possible way, we get three trivial
knots and one trefoil).

In exactly the same way, we can construct the simplest quasidoodle with two
generic triple points, respecting the triangular diagram of Fig.~\ref{md}.
Perturbing it in four different ways, we shall obtain three trivial (equivalent
to a circle) doodles, and one equivalent to Fig.~\ref{merx}a.

(However, A.~B.~Merkov, who discovered this doodle, came to it from very
different considerations.)

\section{Invariants of I-doodles}

In this and the next sections we consider only the immersed curves in $\R^2$.
To any immersion $\phi: S^1 \to \R^2$ there corresponds a map $S^1 \to S^1:$
any point $x \in S^1$ goes to the direction of the tangent vector $\phi'(x).$
Accordingly to S.~Smale \cite{Smale}, this correspondence is a homotopy
equivalence between spaces $I \K \equiv Imm(S^1, \R^2)$ and $C(S^1,S^1).$ These
spaces split into countably many components labeled by the "winding numbers"
(i.e. the indices of corresponding maps $S^1 \to S^1$). Any of these components
is homotopy equivalent to $S^1,$ the homotopy equivalence being provided by the
image of (the tangent direction of $\phi$ at) the distinguished point of $S^1$.

The discriminant $I\Sigma$ in the space $I \K$ is just the intersection of this
space $I \K$ with the discriminant set $\Sigma \subset \K$ considered in the
previous sections. Its resolution $I\Diamond$ is a subset in $\Diamond,$ namely
the complete preimage of $I\Sigma.$

In its decomposition into $\J$-blocks only the $J$-{\em configurations}, i.e.
the $\J$-cliques without multiple points, can take part.

\prop \label{smale} For any connected component $\CC$ of the space $I \K$ and
for any $A$-con\-fi\-gu\-ra\-tion $J$ in $S^1,$ the space of immersions $S^1
\to \R^2$ respecting this configuration and lying in this component is a {\em
path-connected} open submanifold of the space $\chi(J) \subset \K.$ \eprop

This follows easily from the Smale's theorem. \quad $\Box$
\medskip

For any component $\CC,$ denote by $\CC I \Diamond$ and $\CC I \Diamond_i$  the
intersection of the space $I \Diamond $ (respectively, $I \Diamond_i$) with the
preimage of $\CC$ under the projection $\Diamond \to \K$.

\begin{example} The stratum $I\Diamond_2$ is
an open subset in the space of a fiber bundle, almost coinciding with that
considered in Proposition \ref{si2}, with unique difference that its base is
not the entire space $\Psi =(S^1)^3/S(3),$ but its open part $B(S^1,3)$. The
{\em strangeness} is the linking number in $\CC$ with the direct image of the
fundamental cycle of this subset.  The existence of the strangeness as an
integer-valued invariant is due to the fact that this configuration space
$B(S^1,3)$ is orientable.
\end{example}

All other $(\Delta-1)$-dimensional blocks in all spaces $I\Diamond_3 \sm
I\Diamond_2 $ and $I\Diamond_4 \sm I\Diamond_3 $ are the open subsets of
similar blocks considered in the previous section; the corresponding virtual
generators of the group of invariants of I-doodles are shown in
Figs.~\ref{one}b, \ref{two}a and \ref{two}b.

All these generators define elements of corresponding groups $\bar
H_{\Delta-1}(\Diamond_i \sm \Diamond_{i-1}) \equiv E^1_{i,\Delta-i-1}$. For two
generators shown in Fig.~\ref{two}b this follows from the fact that the entire
boundaries of corresponding blocks in $I\Diamond_4 \sm I\Diamond_3 $ are empty.
For two remaining blocks of Fig.~\ref{two}a we should additionally check that
their boundaries in the block corresponding to the $(5)$-configurations are
trivial; the proof of this fact essentially coincides with that of Proposition
\ref{mdcycle}.

So, for any fixed component of the space $I \K$ the domain in the table
$\{E^1_{p,q}\}$ responsible for the calculation of invariants of orders 2, 3
and 4 of I-doodles from this component looks as is shown in Fig.~\ref{ssimm}.

\begin{figure}
\begin{center}
\unitlength=1.00mm \special{em:linewidth 0.4pt} \linethickness{0.4pt}
\begin{picture}(52.00,38.00)
\put(12.00,6.00){\vector(0,1){32.00}} \put(12.00,6.00){\vector(1,0){40.00}}
\put(49.00,3.00){\makebox(0,0)[cc]{$p$}}
\put(6.00,27.00){\makebox(0,0)[cc]{$\Delta-3$}}
\put(6.00,21.00){\makebox(0,0)[cc]{$\Delta-4$}}
\put(6.00,15.00){\makebox(0,0)[cc]{$\Delta-5$}}
\put(17.00,27.00){\makebox(0,0)[cc]{${\bf Z}$}}
\put(17.00,21.00){\makebox(0,0)[cc]{${\bf Z}$}}
\put(27.00,21.00){\makebox(0,0)[cc]{${\bf Z}$}}
\put(27.00,15.00){\makebox(0,0)[cc]{${\bf Z}$}}
\put(37.00,15.00){\makebox(0,0)[cc]{${\bf Z}^5$}}
\put(9.00,35.00){\makebox(0,0)[cc]{$q$}}
\put(17.00,15.00){\makebox(0,0)[cc]{0}} \put(27.00,9.00){\makebox(0,0)[cc]{0}}
\put(27.00,27.00){\makebox(0,0)[cc]{0}} \put(37.00,21.00){\makebox(0,0)[cc]{0}}
\put(17.00,33.00){\makebox(0,0)[cc]{0}} \put(17.00,3.00){\makebox(0,0)[cc]{2}}
\put(27.00,3.00){\makebox(0,0)[cc]{3}} \put(37.00,3.00){\makebox(0,0)[cc]{4}}
\put(12.00,12.00){\line(1,0){33.00}} \put(45.00,18.00){\line(-1,0){33.00}}
\put(12.00,24.00){\line(1,0){33.00}} \put(45.00,30.00){\line(-1,0){33.00}}
\put(22.00,31.00){\line(0,-1){25.00}} \put(32.00,6.00){\line(0,1){25.00}}
\put(42.00,31.00){\line(0,-1){25.00}}
\end{picture}
\end{center}
\caption{Spectral sequence for invariants of orders 2, 3 and 4} \label{ssimm}
\end{figure}

\prop \label{degenss} For any connected component $\CC$ of the space $I \K$,
the fragment of the spectral sequence shown in Fig.~\ref{ssimm} degenerates at
the term $E^1,$ i.e. all its elements extend to well defined Borel--Moore
homology classes of the space $\CC I\Diamond.$ \eprop

\begin{proof}
For the group  $E_{2, \Delta-3}$ this is obvious.

The group $E^1_{2,\Delta-4}$ is generated by the fundamental cycle of the
submanifold in $\CC I\Diamond_2 $, consisting of such pairs of the form \{a
3-configuration $(x,y,z) \in B(S^1,3)$; a map $\phi:S^1 \to \R^2$\} that $x+y+z
\equiv 0 (mod \, 2\pi).$ It is obviously a cycle in entire $\CC I\Diamond $,
let us prove that it is not homologous to zero. As $H^2(\CC) \simeq 0,$ it is
sufficient to construct two 1-dimensional cycles in $\CC \sm I\Sigma,$ defining
the same element in $H^1(\CC)$ but such that some (and then any) 2-chain
realizing the homology between these cycles has nonzero intersection number
with this fundamental cycle.

Consider a map $\phi \in I\Sigma \cap \CC$ with unique generic triple point,
and let $\phi_1, \phi_2$ be two its small nondiscriminant perturbations
resolving this triple point in two different ways, see Fig.~\ref{reidem}c.

For any $i=1,2,$ denote by $[\phi_i]$ the 1-cycle in $\CC \sm \Sigma$ swept out
by all maps obtained from $\phi_i$ by all cyclic reparametrizations of the
issue circle $S^1.$ These two cycles are obviously homologous in $\CC,$ and the
intersection index of such a homology with the above manifold is equal to $\pm
3.$

This proves the assertion of Proposition for cells $E_{2,\Delta-4}$ and
$E_{3,\Delta-4}$.

In particular, there exist two elements of the group $H_0(\CC \sm \Sigma)$,
i.e. two linear combinations of doodles in $\CC$, which cannot be distinguished
by the "strangeness" (generating the group $E^\infty_{2,\Delta-3}$ of second
order 0-cohomology classes), but can be distinguished by the invariant
generating the group $E^\infty_{3,\Delta-4}$. (In accordance with
\S~\ref{merxdood}, we can find these combinations by resolving unique point of
multiplicity 4. Indeed, the linear combination of four doodles, locally
situated as in Fig.~\ref{diffour} and coinciding outside it, provides such a
chain.)

\begin{figure}
\unitlength 0.90mm \linethickness{0.4pt}
\begin{picture}(131.00,22.00)
\put(10.00,10.00){\vector(1,0){20.00}} \put(20.00,0.00){\vector(0,1){20.00}}
\put(7.00,3.00){\vector(1,1){19.00}} \put(42.00,10.00){\vector(1,0){20.00}}
\put(52.00,0.00){\vector(0,1){20.00}} \put(39.00,3.00){\vector(1,1){19.00}}
\put(29.00,3.00){\vector(-1,1){18.00}} \put(59.00,1.00){\vector(-1,1){19.00}}
\put(74.00,10.00){\vector(1,0){20.00}} \put(84.00,0.00){\vector(0,1){20.00}}
\put(106.00,10.00){\vector(1,0){20.00}} \put(116.00,0.00){\vector(0,1){20.00}}
\put(93.00,3.00){\vector(-1,1){18.00}} \put(123.00,1.00){\vector(-1,1){19.00}}
\put(80.00,0.00){\vector(1,1){19.00}} \put(112.00,0.00){\vector(1,1){19.00}}
\put(101.00,10.00){\makebox(0,0)[cc]{{\large $+$}}}
\put(68.00,10.00){\makebox(0,0)[cc]{{\large $-$}}}
\put(36.00,10.00){\makebox(0,0)[cc]{{\large $-$}}}
\put(4.00,10.00){\makebox(0,0)[cc]{{\large $+$}}}
\end{picture}
\caption{The chain of I-doodles nontrivial by the invariant of order 3}
\label{diffour}
\end{figure}
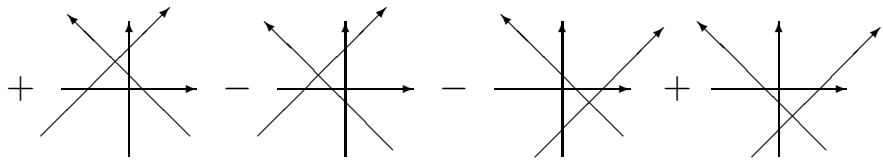

Exactly as above we produce from them an 1-cycle in $\CC \sm \Sigma$,
homologous to zero in $\CC$, and having a nonzero intersection index with the
chain generating the group $E^1_{3,\Delta-5}.$ This proves our Proposition.
\end{proof}

In particular, we have proved that for any component $\CC$ of the space $I \K$
all 7 generators mentioned in Theorem \ref{idoodinv} and shown in
Figs.~\ref{one}, \ref{two} define independent elements of the group $\bar
H_{\Delta-1}(\Sigma \cap \CC).$

Finally, it is obvious that the intersection indices with all these
$(\Delta-1)$-di\-men\-si\-o\-nal Borel--Moore homology classes of the
discriminant define zero elements in the 1-dimensional cohomology group of the
component $\CC$, and hence the linking numbers with them are well defined
invariants of I-doodles.

Theorem \ref{idoodinv} is thus completely proved.

\section{1-dimensional cohomology of the space of immersions
$S^1 \to \R^2$ without points of multiplicity 4}

Define the discriminant $I \Sigma 4 \subset I \K$ as the set of immersions
$\phi: S^1 \to \R^2$ such that images of some 4 different points coincide.

Its resolution $I \Diamond 4$
 is constructed in essentially the same
way as it was done above for the set $I\Sigma.$ In this section we are
interested in the 1-dimensional cohomology classes of the space $I\K \setminus
I\Sigma 4$ or, which is the same, in the $(\Delta-2)$-dimensional Borel--Moore
homology classes of spaces $I\Sigma 4$ or $I\Diamond 4.$

The first nonempty term of this resolution is of filtration 3. This is the
space of an orientable $(\Delta-6)$-dimensional vector bundle over the
configuration space $B(S^1,4).$ This configuration space is non-orientable,
therefore the fundamental cycle of this term defines a class only in the group
$\bar H_{\Delta-2}(I \Sigma_4, \Z_2),$ but not in the integer homology group,
see in Theorem \ref{4ptfree} the statement "$\F_3 \sim \Z_2$ over $\Z_2$".
\medskip

The next term $I {\Diamond 4}_4 \sm I {\Diamond 4}_3$ of our filtration also is
the space of a fiber bundle, whose base is the configuration space $B(S^1,5),$
and the fibers are direct products of stars $\star$ with 5 rays (without
endpoints, which belong to the smaller term of the filtration) and some
(canonically oriented) $(\Delta-8)$-dimensional vector subspaces in $\K$.

The base $B(S^1,5)$ of this bundle is orientable (and diffeomorphic to $S^1
\times \R^4$), and the monodromy over the circle generating the group
$\pi_1(B(S^1,5)) \sim \Z$ acts on the fibration of 5-stars $\star$ by cyclic
permutations of their rays.

Therefore the $(\Delta-2)$-dimensional Borel--Moore homology group of this term
coincides with the subgroup of the group $\bar H_1(\star)$ consisting of
elements invariant under the rotations of these stars. For any coefficient
group $G$ this group is isomorphic to $G^4$. If in $G$ the condition $5a=0$
implies $a=0,$ then its invariant subgroup is trivial; in the case $G=\Z_5$
this group is isomorphic to $\Z_5,$ see statement "$\F_4 \simeq \Z_5$" over
$\Z_5$ of Theorem \ref{4ptfree}.

Finally, consider the term $I{\Diamond 4}_5 \sm I{\Diamond 4}_4$ of our
filtration. It is the space of a fiber bundle over $B(S^1,6),$ whose fiber is
the product of $\R^{\Delta-10}$ and some two-dimensional order complex. This
complex is similar to the one considered in the proof of Proposition
\ref{three_hom}, with unique difference: its vertices correspond to choices of
some 4, 5 or 6 points of our 6, and not of 3, 4 or 5 points of 5. Absolutely as
previously, the two-dimensional cycles of this complex are the linear
combinations of (double) arrows with starts and ends at these 6 points, forming
the cycles (in the usual sense) of the complete graph on these 6 vertices.

However, unlike the case of 5-configurations, the base space $B(S^1,6)$ is
non-orientable. Therefore the $(\Delta-2)$-dimensional Borel--Moore homology
classes of our block are in a one-to-one correspondence with such cycles of the
complete 6-graph, which are {\em anti-invariant} under the cyclic permutations
of its 6 vertices.

The group of such cycles can be easily calculated and is isomorphic to $\Z^2$;
the pictures of its generators are given in Fig.~\ref{hexagon}.

By dimensional reasons, these homology classes of the space $I{\Diamond 4}_5
\sm I{\Diamond 4}_4$ can be extended to these of the space $I{\Diamond 4}_5,$
and hence to the 1-dimensional cohomology classes (of order 5) of the entire
space of immersions $S^1 \to \R^2$ without 4-fold points.

This proves Theorem \ref{4ptfree} of the Introduction.

\end{document}